\documentclass[11pt,oneside,english,jou]{amsart}
\usepackage[T1]{fontenc}
\usepackage[latin9]{inputenc}
\usepackage{verbatim}
\usepackage{amsthm}
\usepackage{amstext}
\usepackage{amssymb}
\usepackage{fancyhdr}
\makeatletter

\providecommand{\tabularnewline}{\\}

\numberwithin{equation}{section}
\numberwithin{figure}{section}
\theoremstyle{plain}
\newtheorem{thm}{Theorem}[section]
  \theoremstyle{definition}
  \newtheorem{defn}[thm]{Definition}
  \theoremstyle{remark}
  \newtheorem{rem}[thm]{Remark}
  \theoremstyle{plain}
  \newtheorem*{thm*}{Theorem}
  \theoremstyle{plain}
  \newtheorem{lem}[thm]{Lemma}
  \theoremstyle{plain}
  \newtheorem{prop}[thm]{Proposition}
  \theoremstyle{plain}
  \newtheorem{cor}[thm]{Corollary}
  \theoremstyle{definition}
  \newtheorem{problem}[thm]{Problem}

\newcommand{\N}{\mathbb{N}}
\newcommand{\Z}{\mathbb{Z}}

\newcommand{\al}{\alpha}

\newcommand{\Pcal}{\mathcal{P}}

\newcommand{\Ucal}{\mathcal{U}}
\newcommand{\D}{\mathcal{D}}
\newcommand{\mcal}{\mathcal{M}}

\newcommand{\pcal}{\mathcal{P}}
\newcommand{\ncal}{\mathcal{N}}
\newcommand{\Acal}{\mathcal{A}}
\newcommand{\acal}{\mathcal{A}}
\newcommand{\Qcal}{\mathcal{Q}}
\newcommand{\qcal}{\mathcal{Q}}
\newcommand{\dcal}{\mathcal{D}}

\newcommand{\ha}{\hat{\alpha}}

\makeatother

\usepackage{babel}

\begin{document}

\title{Minimal hyperspace actions of homeomorphism groups of h-homogeneous
spaces}

\author{Eli Glasner \& Yonatan Gutman}




\thanks{
The first named author's research was supported by Grant No 2006119
from the United States-Israel Binational Science Foundation (BSF)}

\date{\today}

\begin{abstract}
Let $X$ be a h-homogeneous zero-dimensional compact Hausdorff space,
i.e. $X$ is a Stone dual of a homogeneous Boolean algebra. Using
the dual Ramsey theorem and a detailed combinatorial analysis of what
we call \emph{stable collections} of subsets of a finite set, we obtain
a complete list of the minimal sub-systems of the compact dynamical
system $(Exp(Exp(X)),Homeo(X))$, where $Exp(X)$ stands for the hyperspace
comprising the closed subsets of
$X$ equipped with the Vietoris
topology. The importance of this dynamical system stems from Uspenskij's
characterization of the universal ambit of
$G=Homeo(X)$. The results apply
to
$X=C$ the Cantor set,
the generalized Cantor sets $X=\{0,1\}^{\kappa}$
for non-countable cardinals $\kappa$, and to several other spaces. A
particular interesting case is $X=\omega^{\ast}=\beta\omega\setminus\omega$,
where $\beta\omega$ denotes the Stone-$\check{C}$ech compactification
of the natural numbers. This space, called \emph{the corona} or the
\emph{remainder} of $\omega$, has been extensively studied in the
fields of set theory and topology.
\end{abstract}

\maketitle

\section{Introduction}

\subsection{\label{sub:Representative-families}Representative families}

Let $G$ be a (Hausdorff) topological group and $X$ a Hausdorff
compact space.
We consider compact dynamical systems or $G$-spaces which we denote by $(X,G)$.
The general theory of such systems ensures the existence and uniqueness
of a universal $G$-ambit denoted $(\mathcal{S}(G),e_0,G)$.
Here an {\textbf{ambit}} is a $G$-space $(X,x_0,G)$, with a distinguished point
$x_0$ whose orbit is dense in $X$. The universality means that for every ambit
$(X,x_0,G)$ there is a (necessarily unique) homomorphism of pointed dynamical systems $\pi: (\mathcal{S}(G),e_0,G) \to (X,x_0,G)$.
By Zorn's lemma every dynamical system $(X,G)$ admits, at least one, minimal
subsystem $Y \subset X$; i.e. $Y$ is closed and invariant and the only invariant
subsets of $Y$ are $Y$ and $\emptyset$.
It then follows that any minimal subset $M$ of $\mathcal{S}(G)$ is a
{\textbf {universal minimal $G$-system}}, in an obvious sense, and moreover
up to isomorphism this universal system $\mathcal{M}(G)$ is unique.

The {\textbf{enveloping semigroup}} of a dynamical system $(X,G)$,
denoted $E(X)$, is by definition the closure in the compact space $X^X$
of the collection of maps $\{\check{g} : g \in G\}$, where $\check{g}$ is
the element of the group $Homeo(G)$ which corresponds to $g$.
We refer the reader to books on the abstract theory of topological
dynamics for more details. In particular \cite{dV93} is a suitable
source from our point of view.

In \cite{U09}
Uspenskij introduced the following definition

\begin{defn}
A family $\{X_{\alpha} : \,\alpha\in A\}$ of compact G-spaces is \textbf{representative}
if the family of natural maps $\mathcal{S}(G)\rightarrow E(X_{\alpha})$,
where $\mathcal{S}(G)$ is the universal ambit of $G$ and $E(X_{\alpha})$
is the enveloping semigroup of $X_{\alpha}$, separate points of $\mathcal{S}(G)$
(and hence yields an embedding of $\mathcal{S}(G)$ into $\Pi_{\alpha\in A}E(X_{\alpha})$).
\end{defn}

In the same article Uspenskij proved:
\begin{thm}
If $\{X_{\alpha} : \,\alpha\in A\}$ is a representative family of compact
$G$-spaces, the universal minimal compact $G$-space $\mathcal{M}(G)$
is isomorphic (as a $G$-space) to a $G$-subspace of a product $\Pi_{\beta}Y_{\beta}$,
where each $Y_{\beta}$ is a minimal compact $G$-space isomorphic
to a $G$-subspace of some $X_{\alpha}$.
\end{thm}
Denote by $Exp(X)$, the \textit{hyperspace of $X$}, defined to be
the collection all non-empty closed sets of $X$, equipped with the
Vietoris topology. $Exp(X)$ is known to be compact Hausdorff. Let
$G=Homeo(X)$ equipped with the compact-open topology. Notice $Exp(X)$
is a $G$-space. The following is Theorem $4.1$ of Uspenskij \cite{U09}:
\begin{thm}
Let $X$ be a compact space and
$H$ a subgroup of $Homeo(X)$.
The sequence $\{Exp(Exp(X))^{n}\}_{n=1}^{\infty}$ of compact
$H$-spaces is representative.\footnote{By this we mean
$Exp((Exp(X))^{n})$}
\end{thm}
\begin{rem}
We refer to the action of $Homeo(X)$ on $Exp(Exp(X))^{n}$, $n=1,2,\ldots$
as
the
\textbf{hyperspace actions}.
\end{rem}

\subsection{H-homogeneous spaces and homogeneous Boolean algebras}

The following definitions are well know (see e.g. \cite{HNV04} Section
H-4):
\begin{enumerate}
\item A zero-dimensional compact Hausdorff topological space $X$ is called
\textbf{h-homogeneous} if every non-empty clopen subset of $X$ is
homeomorphic to the entire space X.
\item A Boolean algebra $B$ is called \textbf{homogeneous} if for any nonzero
element $a$ of $B$ the relative algebra
$B|a=\{x\in B:\, x\leq a\}$
is isomorphic to $B$.
\end{enumerate}
Using Stone's Duality Theorem (see \cite{BS81} IV$\S4$) a zero-dimensional
compact Hausdorff h-homogeneous space $X$ is the Stone dual of a
homogeneous Boolean
algebra, i.e. any such space is realized as the
space of ultrafilters $B^{\ast}$ over a homogeneous Boolean algebra
$B$ equipped with the topology for which $N_{a}=\{U\in B^{\ast}:\, a\in U\}$,
$a\in B$ is a base. Here are some examples of h-homogeneous spaces
(see \cite{SR89}):
\begin{enumerate}
\item The countable atomless Boolean algebra is homogeneous. It corresponds
by Stone duality to the Cantor set $C=\{0,1\}^{\mathbb{N}}$.
\item Every infinite free Boolean algebra is homogeneous. These Boolean
algebras correspond by Stone duality to the generalized Cantor spaces,
$\{0,1\}^{\kappa}$ , for infinite cardinals $\kappa$.
\end{enumerate}
More examples are discussed in section 1.1 of \cite{GG12}. In next subsection we discuss an especially interesting example:

\subsection{The corona}

For $X$ a Tychonoff space (completely regular Hausdorff space), the {\em Stone-\v{C}ech compactification\/} $\beta X$ of $X$ is a compact Hausdorff space, unique up to homeomorphism, such that $X$
densely embeds in $\beta X$, $X\hookrightarrow \beta X$  and such that the following universal property holds: Any continuous function $\phi:X\rightarrow K$, where $K$ is compact Hausdorff, can be uniquely extended to a continuous function $\tilde{\phi}:\beta X\rightarrow K$.
When $X$ is discrete and in particular
in the case of the integers, $\beta\Z$ has a concrete description as the collection of ultrafilters on $\Z$. The collection $\Ucal=\{U_{A}:A\subset\Z\}$,
where for each $A\subset\Z$ the set $U_{A}$ is the set of ultrafilters
in $\beta\Z$ containing $A$ ($U_{A}=\{p\in\beta\Z:A\in p\}$), forms
a basis for the compact Hausdorff topology on $\beta\Z$. The collection
of {\em fixed\/} ultrafilters; i.e. ultrafilters of the form $p_{n}=\{A\subset\Z:n\in A\}$
for $n\in\Z$, forms an open, discrete, dense subset of $\beta\Z$,
and one identifies this collection with $\Z$. For more details on the Stone-\v{C}ech compactification
see \cite{GJ60}, \cite{E78} and \cite{HS98}.

Given a locally compact Hausdorff space $Y$ it can be shown that $Y$ embeds inside $\beta Y$ as an open dense set. One defines the \textit{corona of $Y$ }(or \textit{remainder
of $Y$)} to be the compact space $\chi(Y)=\beta Y \setminus Y$. Let $\omega$
be the first infinite cardinal which we will identify with $\mathbb{Z}$.
The corona of the integers $\chi(\mathbb{Z})\triangleq\omega^{*}$,
which we will simply call the \textit{corona,} has been extensively
studied in the fields of set theory and logic. An excellent survey
article is \cite{vM84}. The notion of $P$-\textit{points} received
special attention, see \cite{K80} and \cite{BV80}. Specific corona
spaces such as $\chi(\mathbb{Z})$, $\chi(\mathbb{Q})$ and $\chi(\mathbb{R})$
appeared in the now classical monograph on commutative $C^*$-algebras
\cite{GJ60}. The name seems to originate in \cite{GP84}. The non-commutative
analogue of the corona spaces, namely the \textit{corona algebras
}play an important role in the solution of various lifting problems
in the theory of $C^*$-algebras (see \cite{OP89}).

For an infinite subset $A\subset\Z$ let $\hat{A}=\omega^{\ast}\cap{\rm Cls{}_{\beta\Z}(A)}$.
One sees easily that $\hat{A}=\hat{B}$ iff $A\triangle B$ is finite.
The collection $\Ucal=\{\hat{A}:A\subset\Z,\ A\ \text{infinite}\}$
is a basis consisting of clopen sets for the topology of $\omega^{\ast}$.
%
{}For infinite $A\subset\Z$ one has $\hat{A}\simeq\chi(A)$ and moreover
$A\simeq\Z$ implies $\hat{A}\simeq\omega^{\ast}$ using the universal
property of the Stone-\v{C}ech compactification. This shows that $\omega^{\ast}$
is h-homogeneous. Let $P(\omega)$ be the Boolean algebra of all subsets
of $\omega$ and let $fin\subset P(\omega)$ be the ideal comprising
the finite subsets of $\omega$. Define the equivalence relations
$A\sim_{fin}B$, $A,B\in P(\omega)$, if and only if $A\triangle B$
is in $fin$. The quotient Boolean algebra $P(\omega)/fin$ is homogeneous.
This Boolean algebra corresponds by Stone duality to $\omega^{\ast}$.

\subsection{The space of maximal chains}

Let $K$ be a compact Hausdorff space. A subset
$c\subset Exp(K)$
is a \textit{chain} in $Exp(K)$ if for any $E,F\in c$ either $E\subset F$
or $F\subset E$. A chain is \textit{maximal} if it is maximal with
respect to the inclusion relation. One verifies easily that a maximal
chain in $Exp(K)$ is a closed subset of $Exp(K)$, and that $\Phi=\Phi(K)$,
the space of all maximal chains in $Exp(K)$, is a closed subset of
$Exp(Exp(K))$, i.e. $\Phi(K)\subset Exp(Exp(K))$ is a compact space.
Note that a $G$-action on $K$ naturally induces a $G$-action on
$Exp(K)$ and $\Phi$. It is easy to see that every $c\in\Phi$ has
a first element $F$ which is necessarily of the form $F=\{x\}$.
Moreover, calling $x\triangleq r(c)$ the {\em root\/} of the
chain $c$, it is clear that the map $\pi:\Phi\to K$, sending a chain
to its root, is a homomorphism of dynamical systems.

\subsection{The main theorem}

In view of Uspenskij's theorems mentioned in Subsection \ref{sub:Representative-families}
one is naturally interested in classifying the $G$-minimal subspaces
of $\{Exp(Exp(X))^{n}\}_{n=1}^{\infty}$ where $G=Homeo(X)$. The
aim of this work is to accomplish this task for the case of a Hausdorff
zero-dimensional compact h-homogneous space $X$ and $n=1$. This
turns out to be highly nontrivial and therefore hard to generalize
for $n\geq2$. Fortunately the universal minimal space can be calculated
in this case using different methods.
We refer the reader to our paper \cite{GG12}, where we show that
for these spaces $M(G) = \Phi(X)$.
Our main result in the present work is the following:
\begin{thm*}\label{main}
Let $X$ be a Hausdorff zero-dimensional compact h-homogeneous space.
The following list is an exhaustive list of the (closed) $Homeo(X)$ - minimal 
subspaces of $Exp(Exp(X)):$
\begin{enumerate}
\item $\{\{X\}\}$.
\item $\Phi$.
\item $\{\{\{x_{1},x_{2},\ldots,x_{j}\}\}_{(x_{1},x_{2},\ldots,x_{j})\in X^{j}}\}$
($j\in\mathbb{N}$).
\item $\{\{\{x_{1},x_{2},\ldots,x_{j}\},X\}_{(x_{1},x_{2},\ldots,x_{j})\in X^{j}}\}$
($j\in\mathbb{N}$).
\item $\{\{\{x_{1},x_{2},\ldots,x_{j}\},F\}_{(x_{1},x_{2},\ldots,x_{j})\in X^{j},F\in\xi}\}_{\xi\in\Phi}$
($j\in\mathbb{N}$).
\item $\{\{\{x_{1},x_{2},\ldots,x_{j}\}_{(x_{2},\ldots,x_{j})\in X^{j-1}}\}\}_{x_{1}\in X}$
($j\in\mathbb{N}$).
\item $\{\{F\cup\{x_{1},x_{2},\ldots,x_{q}\}\}_{(x_{1},x_{2},\ldots,x_{q})\in X^{q},F\in\xi}\}{}_{\xi\in\Phi}(q\geq1.)$.
\item $\{\{F\cup\{x_{1},x_{2},\ldots,x_{q}\},\{r(\xi),y_{2},\ldots,x_{l}\}\}_{(x_{1},x_{2},\ldots,x_{q})\in X^{q},(y_{2},\ldots,y_{l-1})\in X^{l-1},F\in\xi}\}{}_{\xi\in\Phi}
\newline
(l>q\geq1)$.
\item $\{\{F\cup\{x_{1},x_{2},\ldots,x_{q}\},
\{z_{1},z_{2},\ldots,z_{j}\}\}_{(x_{1},x_{2},\ldots,x_{q})\in X^{q},(z_{1},z_{2},\ldots,z_{j})\in X^{j},F\in\xi}\}{}_{\xi\in\Phi}(q,j\geq1)$.
\item
$\{\{F\cup\{x_{1},x_{2},\ldots,x_{q}\},\{r(\xi),y_{2},\ldots,y_{l}\},
\newline
\{z_{1},z_{2},\ldots,z_{j}\}\}_{(x_{1},x_{2},\ldots,x_{q})\in X^{q},(y_{2},\ldots,y_{l-1})\in X^{l-1},(z_{1},z_{2},\ldots,z_{j})\in X^{j},F\in\xi}\}{}_{\xi\in\Phi}$
\newline
$(l>q\geq1\leq j<l).$
\item $\{\{X,\{x_{1},x_{2},\ldots,x_{j}\}_{(x_{2},\ldots,x_{j})\in X^{j-1}}\}\}_{x_{1}\in X}$
$(j\in\mathbb{N}).$
\item $\{\{\{r(\xi),x_{2},\ldots,x_{j}\},F\}_{(x_{2},\ldots,x_{j})\in X^{j-1},F\in\xi}\}_{\xi\in\Phi}$$(j\in\mathbb{N}).$
\item $\{\{\{y_{1},y_{2},\ldots,y_{j'}\},\{x_{1},x_{2},\ldots,x_{j}\}_{(x_{2},\ldots,x_{j})\in X^{j-1},(y_{1},y_{2},\ldots,y_{j'})\in X^{j'}}\}\}_{x_{1}\in X}$.
\item $\{\{X,\{y_{1},y_{2},\ldots,y_{j'}\},\{x_{1},x_{2},\ldots,x_{j}\}_{(x_{2},\ldots,x_{j})\in X^{j-1},(y_{1},y_{2},\ldots,y_{j'})\in X^{j'}}\}\}_{x_{1}\in X}$.
\item $\{\{\{y_{1},y_{2},\ldots,y_{j'}\},\{r(\xi),x_{2},\ldots,x_{j}\},F\}_{(x_{2},\ldots,x_{j})\in X^{j-1},F\in\xi,(y_{1},y_{2},\ldots,y_{j'})\in X^{j'}}\}_{\xi\in\Phi}$.
\item $\{\{F|\, F\in Exp(X)\}\}$.
\item $\{\{F|\, x\in F\in Exp(X)\}\}_{x\in X}$.
\item $\{\{\{x_{1},x_{2},\ldots,x_{j}\},F\}_{(x_{1},x_{2},\ldots,x_{j})\in X^{j},x\in F\in Exp(X)}\}_{x\in X}$
\textup{$(j\in\mathbb{N}).$}
\end{enumerate}
\end{thm*}
The proof of the theorem is achieved by a detailed combinatorial analysis
of collections of subsets of a finite set,
which we expect will be, in itself, of an independent interest to combinatorists.
We thank Noga Alon for his advise pertaining to some aspects of this analysis.

It is interesting to compare this theorem to similar theorems in \cite{G08}. In that article it is shown that If $X$ is a closed manifold of dimension $2$ or higher, or the
Hilbert cube, then $M$, the space of maximal chains of continua, is
a minimal subspace of $Exp(Exp(X))$ under the action of ${\rm Homeo}(X)$. Further investigating $Exp(M)\subset Exp(Exp(Exp(X)))$ it is shown:

\begin{thm*}
If $X$ is a closed manifold of dimension $3$ or higher, or the
Hilbert cube, then the action of ${\rm Homeo}(X)$ on $Exp(M)$, the space of
non-empty closed subsets of the space of  maximal chains of continua,
has exactly the following minimal subspaces:
\def\theenumi{\arabic{enumi}}\def\labelenumi{\rm(\theenumi)}
\advance\leftmargini by9pt\begin{enumerate}
\item $\{M\}$,
\item $\{M_{x}\}_{x\in X},$ where
$M_{x}=\{c\in M(X) :  \bigcap\{c_{\alpha} : c_{\alpha}\in c\}=\{x\}\}$,
\item $ \{\{c\}\, : \,c\in M\}$.
\end{enumerate}

\end{thm*}

\section{Patterns }

The following section deals with results in combinatorics of finite
sets. In subsequent sections these results will be used in the context
of hyperspace actions.

\subsection{Patterns and partitions}
\begin{defn}
Let $m,n$ be positive integers. A non-empty collection $\Pcal$ of non-empty vectors in $(Exp(\vec{m}))^{n}$,
where $\vec{{m}}=\{1,2,\dots,m\}$ is called an \textbf{$m_{n}$-pattern}.
Thus an $m_n$-pattern has the form
$$
\pcal =\{P_s=(P^1_s, P_s^2, \cdots ,P_s^n):
s=1,2,\dots,t\},
$$
where each $P_s^i$ is a nonempty subset of $\vec{m}$.
We denote the collection of $m_{n}$-patterns by $C_{n}(m)$. The
number of distinct $m_{n}$-patterns is $r=2^{(2^{m}-1)^{n}}-1$ (we exclude
the empty set). Denote by $C_{n}=\bigcup_{m\in\mathbb{N}}C_{n}(m)$, the
collection of \textbf{$n$-dimensional} patterns. Denote by $C(m)=\bigcup_{n\in\mathbb{N}}C_{n}(m)$, the
collection of \textbf{$m$-pattern}.

The idea behind this definition is that patterns represent neighborhoods
of element of the space $Exp(Exp(X))^n$.
It is much easier to think about, or visualize, a $1$-pattern than a
higher order ones. So we suggest that, in the sequel, the reader will
consider, when each new definition is introduced, the one-dimensional case
first.

As a motivating example consider the $m_1$-pattern $\phi_m \in C_1(m)$:
$$
\phi_{m}=\{\{1\},\{1,2\},\dots,\{1,2,\dots,m\}\}.
$$
Note that given a clopen ordered partition $\alpha=(A_{1},A_{2}\ldots,A_{m})$
of the compact zero-dimensional space $X$, the pattern $\phi_m$
can serve as a typical neighborhood $\mathfrak{U}=\mathfrak{U}(\phi_m)$ in the Vietoris topology on $Exp(Exp(X)))$ of a maximal chain $c \in \Phi(x)$.
Explicitly, set
\begin{align*}
\mathfrak{U} & = \langle  \langle A_{1} \rangle,
\langle  A_{1}, A_{2}  \rangle,
\ldots,
\langle  A_{1},A_{2}, \ldots,A_{m}  \rangle \rangle\\
& = \langle \mathcal{U}_1,  \mathcal{U}_2, \dots, \mathcal{U}_m \rangle,
\end{align*}
where for a compact space $K$ and open sets $V_1,\dots,V_k$
we let
$$
\langle V_1, V_2, \dots, V_k \rangle =
\{F \in Exp(K) : \ F \subset \cup_{j=1}^k V_j,\ {\text{\rm and}}\
F \cap V_j \ne \emptyset\ \ \forall \ 1 \le j \le k\}.
$$
In other words, a maximal chain $c \in \Phi(X)$ is in $\mathfrak{U}$
if and only if every
$F \in c$ is in at least one of the sets
$\mathcal{U}_j = \langle A_{1},A_{2}, \ldots,A_{j}  \rangle$ and for every
$1 \le j \le m$ there is at least one $F \in c \cap
\langle A_{1},A_{2}, \ldots,A_{j}  \rangle$
(see Lemma \ref{basis} below).  Pictorially we can think of $c$ as a
chain of closed subsets of $X$ which grows continuously to fill the sets
$A_1$, then $A_1\cup A_2$, etc. and eventually the entire space
$X = A_1\cup A_2 \cup \cdots \cup A_m$.

An ordered partition $\gamma=(C_{1},\ldots,C_{k})$ of $\{1,\ldots,s\}$ into
$k$ nonempty sets is
said to be
\textbf{ naturally ordered} if for every $1\leq i<j\leq k$,
$\min(C_{i})<\min(C_{j})$. We denote by $\Pi\binom{s}{k}$ the collection
of naturally ordered partitions of $\{1,\ldots,s\}$ into $k$ nonempty
sets.
\end{defn}
%
{}
\begin{defn}
Let $\Pcal$ be an $m_{n}$-pattern and $\gamma=(C_{1},\ldots,C_{k})\in\Pi\binom{m}{k}$.
The \textbf{induced $k_{n}$-pattern} $\Pcal_{\gamma}$
is defined as the collection $\Pcal_{\gamma}=\{P_{\gamma}:P\in\Pcal\}$,
where \[
P_{\gamma}=\{(j_{1,}j_{2},\ldots,j_{n}):C_{j_{1}}\times C_{j_{2}}\times\cdots\times C_{j_{n}}\cap P\ne\emptyset\}.\]

Let $\beta=(B_{1},\ldots,B_{s})\in\Pi\binom{k}{s}$ and $\gamma=(C_{1},\ldots,C_{k})\in\Pi\binom{m}{k}$,
we
define the \textbf{amalgamated partition} $\gamma_{\beta}=(G_{1},\ldots,G_{s})\in{m \choose s}$
by:
\[
G_{j}=\bigcup_{i\in B_{j}}C_{i}\]
 Notice $\gamma_{\beta}$ is naturally ordered and $(\Pcal_{\gamma})_{\beta}=\Pcal_{\gamma_{\beta}}.$
\end{defn}

\subsection{Notation}

Define $^{+}:Exp(\vec{m})\rightarrow Exp(\vec{m+1})$ by the mapping
$$
A\mapsto A^+ = \{j+1|\, j\in A\}.
$$
In addition $\emptyset^{+}=\emptyset$.
Define $^{-}:Exp(\vec{m}\setminus\{1\})\rightarrow Exp(\vec{m-1})$
by the mapping
$$
A\mapsto A^- = \{j-1|\, j\in A\}.
$$
For $j\in\vec{m+1}$,
define $D_{j}:Exp(\vec{m})\rightarrow Exp(\vec{m+1})$ by the mapping
$$
A\mapsto (A\cap\vec{j-1})\cup(A\setminus\vec{j-1})^{+}.
$$

For $\gamma\in\Pi\binom{m+1}{m},\ \gamma = (C_1,C_2,\dots,C_m)$,
denote by
$p_{\gamma}: Exp(\vec{m+1}) \rightarrow Exp(\vec{m})$ the mapping
$A\mapsto A_{\gamma}$, where, as above,
$A_\gamma =\{j : C_j \cap A \ne \emptyset\}$.

Let $i,j\in\vec{m+1}$ with $i<j.$ Define:
\[
\gamma_{i,j}^{m+1}=\gamma_{i,j}=
(\{1\},\{2\},\ldots,\{i-1\},\{i,j\},\{i+1\},\ldots,\{j-1\},\{j+1\},\ldots,\{m+1\})
\]
Notice $\gamma_{i,j}^{m+1}\in\Pi\binom{m+1}{m}$. For $P\in Exp(\vec{m})$,
with $m\in P$ we introduce the notation: $P_{m+1}=P\cup\{m+1\}$
and $\hat{P}=P_{m+1}\setminus\{m\}$. Notice $p_{\gamma_{m,m+1}}^{-1}(P)=\{P,\hat{P},P_{m+1}\}$.

\subsection{The standard Patterns}

%
{}
\begin{defn}\label{standard}
Let $\vec{m}=\{1,\ldots,m\}$ and $\vec{0}=\emptyset$. For $1\leq i_{1}<i_{2}<\cdots<i_{l}\leq m$
let $I_{i_{1},i_{2},\ldots,i_{l}}^{m}=\vec{m}\setminus\{i_{1},i_{2},\ldots,i_{l}\}$ and
define $e(I_{i_{1},i_{2},\ldots,i_{l}}^{m})=i_{1}$. The indices $i_{1},i_{2},\ldots,i_{l}$
are referred to as the \textbf{holes} of $I_{i_{1},i_{2},\ldots,i_{l}}^{m}.$
For $0\leq h\leq m$, $1\leq l\leq m$, $h+l\le m+1$ define $H_{h,l}^{m}=\{I_{d_{1},d_{2},\ldots,d_{h}}^{m}\}_{l\leq d_{1}<d_{2}<\cdots<d_{h}}$
and $EH_{h,l}^{m}=\{I_{d_{1},d_{2},\ldots,d_{h}}^{m}\}_{l=d_{1}<d_{2}<\cdots<d_{h}}$,
where we use the convention $H_{0,l}^{m}=EH_{0,l}^{m}=\vec{m}$.

\label{def:The standard Examples}The following $m_{1}$-patterns
are called
{\textbf {standard}} :
\begin{enumerate}
\item $\{\vec{m}\}$ .
\item $\phi_{m}=\{\{1\},\{1,2\},\dots,\vec{m}\}$.
\item \label{enu:Acal_j_m}For every $1\le j\le m$ the collection $\Acal_{j,m}$
of all subsets of $\vec{m}$ of cardinality $\le j$.
\item \label{enu:Acal_j_m_cup_vec(m)}For every $1\le j\le m-2$ the collection
$\Acal_{j,m}\cup\{\vec{m}\}$.
\item For every $1\le j\le m-2$ the collection $\Acal_{j,m}\cup\phi_{m}$.
\item \label{enu:case_A^1_j,m}For every $1\le j\le m$ the collection $\Acal_{j,m}^{1}$
of all subsets of $\vec{m}$ of cardinality $\le j$ containing $1$.
\item \label{enu:acal^1_ncal_vec(m)} The collection $\Acal_{m-2,m}^{1}\cup\ncal\cup\{\vec{m}\}$
for $\emptyset\neq\ncal\subsetneq H_{1,2}^{m}$ ($\ncal=\emptyset$
corresponds to case (\ref{enu:Acal_one_j_m_cup_vec(m)}) and the case
$\ncal=H_{1,2}^{m}$ corresponds to case (\ref{enu:case_A^1_j,m}))
\item $\Acal_{m-2,m}\cup\ncal$ for $\ncal\subset H_{1,1}^{m}$, $\ncal\neq H_{1,2}^{m}$
with $|\ncal|=m-1$ ($\ncal=H_{1,2}^{m}$ corresponds to case (\ref{enu:Acal^1_j_m_cup_Acal_j'})).
\item \label{enu:Acal_m-2,m_cup_ncal_cup_vec(m)}Let $\emptyset\neq\ncal\subsetneq H_{1,1}^{m}$.
$\ncal\neq H_{1,2}^{m}$. The collection $\Acal_{m-2,m}\cup\ncal\cup\{\vec{m}\}$
($\ncal=\emptyset$ corresponds to case (\ref{enu:Acal_j_m_cup_vec(m)}), $\ncal=H_{1,1}^{m}$ corresponds to case (\ref{enu:Acal_j_m})and $\ncal=H_{1,2}^{m}$ corresponds to case (\ref{enu:Acal^1_j_m_cup_Acal_j'})).
\item \label{enu:dcal^m_s_r}For every $2\leq r<r+1<s<m$ the collection
$\dcal_{r,s}^{m}\triangleq\Acal_{m-r-1,m}^{1}\cup\bigcup_{h=1}^{r}H_{h,s-h+1}^{m}\cup\{\vec{m}\}$
($s=m$ corresponds to case (\ref{enu:Acal_j_m_phi_m}), $s=r+1$
corresponds to $\dcal_{r-1,s}^{m}$ and $r=1$ corresponds to case
(\ref{enu:acal^1_ncal_vec(m)}) ).
\item For every $2\leq r<r+1<s<m$ and $1\leq j\leq m-r-1$, $\dcal_{r,s}^{m}\cup\Acal_{j,m}$.
\item \label{enu:Acal_one_j_m_cup_vec(m)}For every $1\le j\le m-2$ the
collection $\Acal_{j,m}^{1}\cup\{\vec{m}\}.$
\item \label{enu:Acal_j_m_phi_m}For every $1\le j\le m-2$ the collection
$\Acal_{j,m}^{1}\cup\phi_{m}$.
\item \label{enu:Acal^1_j_m_cup_Acal_j'}For every $2\le j\le m$ and $j'<j$
the collection $\Acal_{j,m}^{1}\cup
\Acal_{j',m}$.
\item For every $2\le j\le m-2$ and $j'<j$ the collection $\Acal_{j,m}^{1}\cup\Acal_{j',m}\cup\{\vec{m}\}$
($j=m-1,m$ correspond to case (\ref{enu:Acal^1_j_m_cup_Acal_j'})).
\item For every $2\le j\le m-2$ and $j'<j$ the collection $\Acal_{j,m}^{1}\cup\Acal_{j',m}\cup\phi_{m}$
($j=m-1,m$ correspond to case (\ref{enu:Acal^1_j_m_cup_Acal_j'})).
\item For every $1\le j<m-2$ the collection $\Acal_{j,m}\cup\Acal_{m-2,m}^{1}\cup\ncal\cup\{\vec{m}\}$
where $\emptyset\neq\ncal\subsetneq H_{1,2}^{m}$.
\end{enumerate}
\end{defn}
Note:
\begin{enumerate}
\item $\dcal_{r_{1},s_{1}}^{m}\cup\dcal_{r_{2},s_{2}}^{m}=\dcal_{\min\{r_{1},r_{2}\},\min\{s_{1},s_{2}\}}^{m}.$
\item $D_{s,s+1}^{m}=D_{s-1,s+1}^{m}$.
\end{enumerate}

\subsection{Stable patterns}
The notion of a stable pattern which we are about to define is of crucial
importance for our analysis. We surmise that it may be relevant for other
problems in the combinatorics of finite sets.
\begin{defn}\label{SP}
An $m_{n}$-pattern $\Pcal$ is said to be $k$\textbf{-stable} if
for every partition $\alpha\in\Pi\binom{m}{k'}$ for $2\leq k'\leq k$
the induced $k'_{n}$-pattern $\Pcal_{\alpha}$ is a constant pattern
(i.e. it does not depend on $\alpha$).
\noindent Denote $SP_{n}(m)=\{\Pcal : \,\Pcal\textrm{ is an \ensuremath{m}-stable \ensuremath{m_{n}}-pattern}\}$.
\end{defn}

As an example the reader is advised to check that the $m_1$-pattern
$\phi_m$ is $m$-stable.
\begin{lem}
\label{lem:D_j}Let $1\leq i<j\leq m+1$ and $Q\subset\vec{m}$ with
$i\notin Q$ then $p_{\gamma_{i,j}}^{-1}(Q)=\{D_{j}(Q)\}$.
\end{lem}
\begin{proof}
Trivial. %
{} \end{proof}

\begin{lem}
\label{lem:m partitions are enough}Let $m\in\mathbb{N}$ and let
$\Pcal$ be an $(m+1)$-pattern. If there exists an $m$-stable $m$-pattern
$\Qcal$ so that for every $\gamma\in\Pi{m+1 \choose m}$, $\Pcal_{\gamma}=\Qcal$
then $\Pcal$ is $(m+1)$-stable.\end{lem}
\begin{proof}
For every $\alpha,\alpha'\in\Pi{m+1 \choose k}$, there exist $\gamma,\gamma'\in\Pi{m+1 \choose m}$
and $\beta,\beta'\in\Pi{m \choose k}$ so that $\Pcal_{\alpha}=(\Pcal_{\gamma})_{\beta}=\Qcal_{\beta}=\Qcal_{\beta'}=(\pcal_{\gamma'})_{\beta'}.$ \end{proof}

\begin{lem}\label{lemma:trivial}
Let $m\in\mathbb{N}$ and let $\Pcal$ be an
$(m+1)$-pattern. Let $\pi\in\Pi{m+1 \choose m}$ and $\Qcal$ an
$m$-pattern so that $\Pcal_{\pi}=\Qcal$ then $\Pcal=\bigcup_{Q\in\Qcal}p_{\pi}^{-1}(Q)\cap\Pcal$.\end{lem}
\begin{proof}
Trivial. \end{proof}
\begin{lem}
\label{lemma:p^-1=00003D00003D00003D00003D00003D00003D00003D1}Let
$m\in\mathbb{N}$ and let $\Pcal$ be an $(m+1)$-pattern. Let $\gamma\in\Pi{m+1 \choose m}$
and $\Qcal$ an $m$-pattern so that $\Pcal_{\gamma}=\Qcal$. Let
$Q\in\Qcal$ and assume $p_{\gamma}^{-1}(Q)=\{A\}$, then $A\in\Pcal$. \end{lem}
\begin{proof}
Trivial. \end{proof}

\begin{lem}
\label{lem:H is in pcal}Let $\pcal$ be an $(m+1)$-stable
$(m+1)_{1}$-pattern.
Let $1\leq h\leq m$, $1\leq l\leq m$, $h+l\le m+1$ and $\pi=\gamma_{m,m+1}^{m+1}\in\Pi{m+1 \choose m}$.
If $\pcal_{\pi}\cap EH_{h,i}^{m}=\emptyset$ for all $i<l$, and $\pcal_{\pi}\cap EH_{h,l}^{m}\neq\emptyset$
then $\pcal\cap H_{h+1,1}^{m+1}=H_{h+1,l}^{m+1}$ and $\pcal_{\pi}\cap H_{h,1}^{m}=H_{h,l}^{m}$.\end{lem}

\begin{proof}
Let $Q\in\pcal_{\pi}\cap EH_{h,l}^{m}$
which is
non-empty by assumption. Our
first goal is to show that $H_{h+1,l}^{m+1}\subset\pcal.$ Let $R=I_{s_{1},s_{2},\ldots,s_{h+1}}^{m+1}\in H_{h+1,l}^{m+1}.$
We will show using induction that for any $l-1\leq j\leq s_{h}$ there
exists $P=P(j)\in\pcal_{\pi}\cap EH_{h,l}^{m}$ so that $P\cap\vec{j}=R\cap\vec{j}$
and $e(P)\leq e(R)$ (obviously if the statement is true one can choose
$P(s_{h})$ for all $j$ but this can be concluded only after the
induction is carried through). First we verify the base case by choosing
$P=Q$ and noticing trivially that $P\cap\vec{l-1}=R\cap\vec{l-1}=\vec{l-1}$
and $l=e(P)\leq e(R)$ . Secondly let $l-1\leq j\leq s_{h}-1$ and
assume there exist $P\in\pcal_{\pi}\cap EH_{h,l}^{m}$ so that $P\cap\vec{j}=R\cap\vec{j}$
and $e(P)\leq e(R)$. We will prove there exists $P'\in\pcal_{\pi}\cap EH_{h,l}^{m}$
so that $P'\cap\vec{j+1}=R\cap\vec{j+1}$ and $e(P')\leq e(R)$. If
$P\cap\vec{j+1}=R\cap\vec{j+1}$ we are done. Assume $P\cap\vec{j+1}\neq R\cap\vec{j+1}$.
We distinguish between several cases. We repeatedly use the fact that
$D_{q}(P)\in\Pcal$ for any $e(P)<q\leq m+1$ as seen by Lemma \ref{lem:D_j}.
\begin{itemize}
\item $j+1\notin R$ and $j+1\in P$. As $R$ has at most $h-1$ holes in
$\vec{j},$ so does $P$ in $\vec{j+1}$. Therefore there exists $k>j+1$
so that $k\notin P$. Moreover by assumption $e(P)\leq e(R)\leq j+1$
and as $j+1\in P$ then $e(P)<j+1$. Define $P'=\pi_{j+2,k+1}(D_{j+1}(P))$
($P'$ is constructed by adding a hole at $j+1$ and canceling the
hole at $k$). Clearly $P'\in\pcal_{\pi}\cap EH_{h,l}^{m}$ and $e(P')=e(P)\leq e(R)$.
\item $j+1\in R$, $j+1\notin P$, $e(P)\leq j$ and $\exists k>j+1$, $k\in P$.
Define $P'=\pi_{j+1,k}(D_{m+1}(P))$ ($P'$ is constructed by canceling
the hole at $j+1$ and adding a hole at $m$). Clearly $P'\in\pcal_{\pi}\cap EH_{h,l}^{m}$.
As $P$ has a hole in $\vec{j}$ so does $P'$ and we have $e(P')=e(P)\leq e(R)$.
\item $j+1\in R$, $e(P)=j+1$ (equivalent to $j+1\notin P$, $e(P)>j$
and implies $e(R)\geq j+2$) and $j+2\in P$. Define $P'=\pi_{1,j+1}(D_{j+3}(P))$
($P'$ is constructed by canceling the hole at $j+1$ and adding a
hole at $j+2$) Clearly $P'\in\pcal_{\pi}\cap EH_{h,l}^{m}$. In addition
notice $e(P')=j+2\leq e(R)$.
\item $j+1\in R$, $e(P)=j+1$ (equivalent to $j+1\notin P$, $e(P)>j$
and implies $e(R)\geq j+2$), $j+2\notin P$ and $\exists k>j+2$,
$k\in P$. $P'=\pi_{j+1,k}(D_{m+1}(P))$ ($P'$ is constructed by
canceling the hole at $j+1$ and adding a hole at $m$). Clearly $P'\in\pcal_{\pi}\cap EH_{h,l}^{m}$.
In addition notice $e(P')=j+2\leq e(R)$.
\item $j+1\in R$, $j+1\notin P$, $e(P)\leq j$ and $\sim\exists k>j+1$,
$k\in P$. Notice that by assumption $j+1\leq s_{h}$. However as
$j+1\in R$, we conclude $j+1<s_{h}$ which implies $j+2<s_{h+1}\le m+1$.
Define $P'=\pi_{j+2,m+1}(D_{j+1}(P))$ ($P'$ is constructed by canceling
the hole at $j+1$ and adding a hole at $j+2$) Clearly $P'\in\pcal_{\pi}\cap EH_{h,l}^{m}$.
In addition notice $e(P')=j+2\leq e(R)$.
\item $j+1\in R$, $j+1\notin P$, $e(P)>j$ and $\sim\exists k>j+1$, $k\in P$.
As $j\geq l-1$ and $e(P)=l$, we conclude $j=l-1$. As $\sim\exists k>j+1$,
$k\in P$ we must have $P=I_{l,l+1,\ldots,m}^{m}$. This implies $h=m-l+1$.
In this case $H_{h+1,l}^{m+1}=\{I_{l,l+1,\ldots,m+1}^{m+1}\}.$ As
$I_{l,l+1,\ldots,m+1}^{m+1}=D_{m+1}(I_{l,l+1,\ldots,m}^{m})$, we
have $H_{h+1,l}^{m+1}\subset\pcal$ and we can stop the induction.
\end{itemize}
At the end of the induction we have either proven $H_{h+1,l}^{m+1}\subset\pcal$
or shown that for all $R=I_{s_{1},s_{2},\ldots,s_{h+1}}^{m+1}\in H_{h+1,l}^{m+1}$
there exists $P\in\pcal_{\pi}\cap EH_{h,l}^{m}$ so that $P\cap\vec{s_{h}}=R\cap\vec{s_{h}}$
and $e(P)\leq e(R)$. This implies $R=D_{s_{h+1}}(P)$ and therefore
we have $R\in\Pcal$. We can thus finally conclude $H_{h+1,l}^{m+1}\subset\pcal.$
This implies $H_{h,l}^{m}\subset\pcal_{\pi}.$ Indeed let $P\in H_{h,l}^{m}$,
then $Q=D_{m+1}(P)\in H_{h+1,l}^{m+1}$ and $P=\pi(Q)$. By assumption
$\pcal_{\pi}\cap EH_{h,i}^{m}=\emptyset$ for all $i<l$, we can therefore
conclude $\pcal_{\pi}\cap H_{h,1}^{m}=H_{h,l}^{m}$. Assume for a
contradiction $A\in\pcal\cap EH_{h+1,i}^{m+1}$ for some $i<l$. Select
$j>i$ so that $j\notin A$ (such
$j$ exists as $(h+1)\geq2$).
Notice $p_{\gamma_{i,j}}(A)\in EH_{h,j}^{m}$ which is a contradiction.
We conclude that $\pcal\cap H_{h+1,1}^{m+1}=H_{h+1,l}^{m+1}$.\end{proof}

\begin{lem}
\label{lem:pcal_j,m+1 is contained}Let $\pcal$ be a $(m+1)$-stable
$(m+1)_{1}$-pattern. Assume $j<m$. If $\acal_{j,m}\subset\pcal_{\pi}$
($\acal_{j,m}^{1}\subset\pcal_{\pi}$) then $\acal_{j,m+1}\subset\pcal$
($\acal_{j,m+1}^{1}\subset\pcal$ respectively).\end{lem}
\begin{proof}
Let $Q\in\acal_{j,m+1}$. Let $1\leq i<k\leq m+1$, so that $i,k\notin Q$.
Notice $|p_{\gamma_{i,k}}(Q)|=|Q|$ and therefore $p_{\gamma_{i,k}}(Q)\in\acal_{j,m}$.
As $p_{\gamma_{i,k}}^{-1}(p_{\gamma_{i,k}}(Q))=\{Q\}$ we have $Q\in\Pcal$.
The proof for $\acal_{j,m+1}^{1}$ is similar. \end{proof}
\begin{lem}
\label{lem:H^m+1_1_1_in_pcal}Let $\pcal$ be a $(m+1)$-stable $(m+1)_{1}$-pattern.\textup{
Assume }$\pcal_{\pi}\cap H_{1,1}^{m}=H_{1,e}^{m}$\textup{ where $e\geq3$,
then }$\pcal\cap H_{1,1}^{m+1}=H_{1,e+1}^{m+1}$\textup{. If in addition
$\vec{m}\in\pcal_{\pi}$, then $\vec{m+1}\in\pcal$.}\end{lem}
\begin{proof}
Denote $\mcal=\{I_{k}^{m+1} : \, I_{k}^{m+1}\in\Pcal\}.$ Choose $e\leq k\leq m$
and notice that $p_{\gamma_{1,2}}^{-1}(I_{k}^{m})=\{I_{k+1}^{m+1},I_{1,k+1}^{m+1},I_{2,k+1}^{m+1}\}$.
Clearly $I_{1,k+1}^{m+1}\notin\Pcal$ as $1\notin p_{\gamma_{2,3}}(I_{1,k+1}^{m+1})$.
Also $I_{2,k+1}^{m+1}\notin\Pcal$ as $p_{\gamma_{k,k+1}}(I_{2,k+1}^{m+1})
\allowbreak =I_{2}^{m}\notin\Pcal_{\pi}$.
Conclude $H_{1,e+1}^{m+1}\subset\mcal$. Let $k'<e$, then $p_{\gamma_{m,m+1}}(I_{k'}^{m+1})=I_{k'}^{m}\notin\Pcal_{\pi}$.
Conclude $\mcal=H_{1,e+1}^{m+1}$. If $\vec{m}\in\pcal_{\pi}$, then
as $\vec{m}\notin p_{\gamma_{e-1,e}}(\mcal)$ we conclude $\vec{m+1}\in\Pcal$. \end{proof}

\begin{thm}
\label{thm:Standard_Examples_are_Stable}Let $m\geq2$. The standard
patterns are $m$-stable $m_{1}$-patterns and in particular for all
$\gamma\in\Pi{m+1 \choose m}$:\end{thm}
\begin{enumerate}
\item $\{\vec{m}\}_{\gamma}=\{\vec{m-1}\}.$
\item $(\phi_{m})_{\gamma}=\phi_{m-1}.$
\item $(\Acal_{m,m})_{\gamma}=\Acal_{m-1,m-1}$.
\item $(\Acal_{m,m}^{1})_{\gamma}=\Acal_{m-1,m-1}^{1}$.
\item \label{ite:Acal}$(\Acal_{j,m})_{\gamma}=\Acal_{j,m-1}$ for $1\leq j\leq m-1$.
\item \label{ite:Acal1}$(\Acal_{j,m}^{1})_{\gamma}=\Acal_{j,m-1}^{1}$
for $1\leq j\leq m-1$.
\item \label{ite:long_I_list}$(\dcal_{r,s}^{m})_{\gamma}=\dcal_{r-1,s-1}^{m-1}$
for $2\leq r<r+1<s<m$.
\item $(\Acal_{m-2,m}^{1}\cup\ncal\cup\{\vec{m}\})_{\gamma}=\acal_{m-1,m-1}^{1}$
where $\ncal\subset H_{1,2}^{m}$.
\item $(\Acal_{m-2,m}\cup\ncal\cup\{\vec{m}\})_{\gamma}=\acal_{m-1,m-1}$
for $\ncal\subset H_{1,1}^{m}$.
\item $(\Acal_{m-2,m}\cup\ncal)_{\gamma}=\acal_{m-1,m-1}$ for $\ncal\subset H_{1,1}^{m}$
with $|\ncal|=m-1$.
\end{enumerate}

\begin{proof}
$ $
\begin{enumerate}
\item Trivial.
\item Trivial.
\item Trivial.
\item Trivial.
\item Fix $\alpha=\gamma_{i,k}$ for some $1\leq i<k\leq m$. Trivially
$(\Acal_{j,m})_{\alpha}\subset\Acal_{j,m-1}$. Let $Q\in\Acal_{j,m-1}$.
If $i\notin Q$, then by Lemma \ref{lem:D_j} $p_{\alpha}^{-1}(Q)=\{D_{k}(Q)\}$
and clearly $D_{k}(Q)\in\Acal_{j,m}$ as $|D_{k}(Q)|=|Q|$. If $i\in Q$,
then $p_{\alpha}^{-1}(Q)=\{D_{k}(Q),D_{k}(Q)\cup\{k\},D_{k}(Q)\cup\{k\}\setminus\{i\}\}$
and again it is enough to note that $D_{k}(Q)\in\Acal_{j,m}$.
\item Similar to
the proof of (\ref{ite:Acal}).
\item We start by proving that for all $\gamma\in\Pi{m+1 \choose m}$, $(\dcal_{r,s}^{m})_{\gamma}\subset\dcal_{r-1,s-1}^{m-1}$,
where: \[
(\dcal_{r,s}^{m})_{\gamma}=(\Acal_{m-r-1,m}^{1}\cup\bigcup_{l=1}^{r}H_{l,s-l+1}^{m}\cup\{\vec{m}\}){}_{\gamma}\]
 \[
\dcal_{r-1,s-1}^{m-1}=\Acal_{m-r-1,m-1}^{1}\cup\bigcup_{l=1}^{r-1}H_{l,s-l}^{m-1}\cup\{\vec{m-1\}}\]
This follows from $(\Acal_{m-r-1,m}^{1})_{\gamma}=\Acal_{m-r-1,m-1}^{1}$,
(as trivially $m-r-1<m$), $\{\vec{m\}}_{\gamma}=\{\vec{m-1\}}$ and
from $(H_{l,s-l+1}^{m})_{\gamma}\subset H_{l-1,s-l+1}^{m-1}\cup H_{l,s-l}^{m-1}$
for $l=1,\ldots,r$. To prove $(\dcal_{r,s}^{m})_{\gamma}\supset\dcal_{r-1,s-1}^{m-1}$,
fix $\gamma=\gamma_{i,j}$ and notice that for $P\in H_{l,s-l}^{m-1}$,
one has $p_{\gamma}(D_{j}(P))=P$ and $D_{j}(P)\in H_{l,s-l+1}^{m}$.
\item Trivial.
\item Trivial.
\item Notice that as $|\ncal|=m-1$, for all $\gamma=\gamma_{i,j}$, there
exist $P\in\ncal$ so that $i\notin P$ or $j\notin P$, which implies
that $\vec{m}\in(\ncal)_{\gamma}$. The rest of the proof is trivial.\end{enumerate}
\end{proof}
Our next goal is to show that, in fact, the standard patterns are
the only $m_{1}$-patterns which are $m$-stable
(Theorem \ref{thm: standard are the only stable}).
We begin by analyzing the $3$-patterns.

\begin{prop}
\label{thm:3 stable patterns}The $3$-stable $3_{1}$-patterns are
standard. \end{prop}

\begin{proof}
We enumerate all $3$-stable $3_{1}$-patterns. Denote $\alpha=\gamma_{1,2}^{3}$,
$\beta=\gamma_{1,3}^{3}$ and $\pi=\gamma_{2.3}^{3}$. Assume $\pcal$
is a $3$-stable $3_{1}$-pattern. Obviously $\pcal_{\pi}$ is one
of the
seven
$2$-patterns. We analyze the different cases and show
$\pcal$ must be
standard:
\begin{enumerate}
\item $\pcal{}_{\pi}=\{\{1\}\}=\acal_{1,2}^{1}$. Notice $p_{\pi}^{-1}(\{1\})=\{\{1\}\}$
and conclude by Lemma \ref{lemma:trivial} $\pcal=\acal_{1,3}^{1}$.
\item $\pcal_{\pi}=\{\{2\}\}$. Notice $p_{\beta}^{-1}(\{2\})=\{\{2\}\}$. Conclude $\pcal=\{\{2\}\}$. However $(\{\{2\}\})_{\alpha}=\{\{1\}\}$.
Contradiction.
\item $\pcal_{\pi}=\{\{1,2\}\}=\{\vec{2}\}$. Using case (\ref{enu:vec(m)})
of Theorem \ref{thm: standard are the only stable} which holds true
for $m\geq2$, we conclude $\pcal=\{\vec{3}\}$
\item $\pcal_{\pi}=\{\{1\},\{2\}\}=\acal_{1,2}$. Using case (\ref{case:pcal=00003D00003D00003D00003D00003D00003DAcal_j_m}a.)
of Theorem \ref{thm: standard are the only stable} which holds true
for $m\geq2$, we conclude $\pcal=\acal_{1,3}$.
\item $\pcal_{\pi}=\{\{1\},\{1,2\}\}=\phi_{2}$. Notice $p_{\pi}^{-1}(\{1,2\})=\{\{1,2\},\{1,3\},\{1,2,3\}\}$.
We analyze all $\qcal\in p_{\pi}^{-1}(\phi_{2})$
in the following table:\\
 \\
 \begin{tabular}{|c|c|c|}
\hline
$\qcal$  & Stable  & Identification / Reason for not being stable\tabularnewline
\hline
\hline
$\{\{1\},\{1,2\}\}$  & No.  & $\qcal{}_{\alpha}=\{\{1\}\}\neq Q_{\pi}$\tabularnewline
\hline
$\{\{1\},\{1,3\}\}$  & No.  & $\qcal{}_{\beta}=\{\{1\}\}\neq Q_{\pi}$\tabularnewline
\hline
$\{\{1\},\{1,2,3\}\}$  & Yes.  & $\acal_{1,3}^{1}\cup\{\vec{3}\}$\tabularnewline
\hline
$\{\{1\},\{1,2\},\{1,3\}\}$  & Yes.  & $\Acal_{2,3}^{1}$\tabularnewline
\hline
$\{\{1\},\{1,2\}\{1,2,3\}\}$  & Yes.  & $\phi_{3}$\tabularnewline
\hline
$\{\{1\},\{1,3\}\{1,2,3\}\}$  & Yes.  & $\Acal_{1,3}^{1}\cup\ncal\cup\{\vec{3}\}$, where $\ncal\subset H_{1,2}^{3}$.\tabularnewline
\hline
$\{\{1\},\{1,2\},\{1,3\}\{1,2,3\}\}$  & Yes.  & $\acal_{3,3}^{1}$\tabularnewline
\hline
\end{tabular}\\
 \\

\item $\pcal_{\pi}=\{\{2\},\{1,2\}\}$. Notice $p_{\beta}^{-1}(\{2\})=\{\{2\}\}$.
Conclude $\{2\}\in\Pcal$. However $p_{\alpha}(\{2\})=\{1\}\notin\pcal_{\pi}$.
Contradiction.
\item $\pcal_{\pi}=\{\{1\},\{2\},\{1,2\}\}=\acal_{2,2}.$ Using case (\ref{enu:Acal_j_m}b.)
of Theorem \ref{thm: standard are the only stable} which holds true
for $m\geq2$, we conclude $\Pcal=\Acal_{1,3}\cup\ncal'$ for $\ncal'\subset H_{1,1}^{3}$
with $|\ncal'|=2$ or $\pcal=\Acal_{2,3}$ or $\Pcal=\Acal_{1,3}\cup\ncal\cup\{\vec{3}\}$,
for $\ncal\subset H_{1,1}^{3}$.
\end{enumerate}
\end{proof}

\begin{thm}
\label{thm: standard are the only stable} The standard patterns are
the only $m_{1}$-patterns which are $m$-stable.\end{thm}
\begin{proof}
By Theorem \ref{thm:Standard_Examples_are_Stable} the standard patterns
are $m$-stable. We prove by induction on $m$ that the standard patterns
are the only $m_{1}$-patterns which are $m$-stable. The case $m=3$
is proven in Theorem \ref{thm:3 stable patterns}. Assume the theorem
is true for $m\geq3$, we prove it for $m+1$. Let $\pi=\gamma_{m,m+1}^{m+1}$.
Let $\Pcal$ be an
$(m+1)$-stable $(m+1)_{1}$-pattern. By the
induction assumption $\Pcal_{\pi}$ must be standard. We analyze the
different cases in order to prove $\Pcal$ is standard. Note that
if $P\in\Pcal_{\pi}$ and $m\notin P$, then $p_{\pi}^{-1}(P)=\{P\}$
and therefore we will be mainly analyzing $P\in\Pcal_{\pi}$ with
$m\in P$.
\begin{enumerate}
\item \label{enu:vec(m)}$\Pcal_{\pi}=\{\vec{m}\}$. Recall $p_{\pi}^{-1}(\vec{m})=\{\vec{m+1},\vec{m},\hat{m}\}$.
We claim $\vec{m}\notin\Pcal.$ Indeed let $\alpha=\gamma_{1,m}^{m+1}$
and observe that $\vec{m}_{\alpha}=\vec{m-1}\notin\Pcal_{\pi}$ (recall
$m\geq2$). Similarly let $\beta=\gamma_{1,m+1}^{m+1}$ and observe
that $\hat{m}_{\alpha}=\vec{m-1}\notin\Pcal$ (recall $m\geq2$).
We conclude $\Pcal=\{\vec{m+1}\}$ using Lemma \ref{lemma:trivial}.
\item \label{case:pcal=00003D00003D00003D00003D00003D00003Dphi_m}$\Pcal_{\pi}=\phi_{m}$.
Recall $p_{\pi}^{-1}(\vec{m})=\{\vec{m+1},\vec{m},\hat{m}\}$. We
claim $\hat{m}\notin\Pcal.$ Indeed let $\alpha=\gamma_{1,2}^{m+1}$
and observe $\hat{m}_{\alpha}=\{1,2,\ldots,m-2,m\}\notin\Pcal_{\pi}$(recall
$m-2\geq1$). We claim $\vec{m}\notin\Pcal$. Indeed in that case
$\vec{m-1}\notin\Pcal_{\alpha}.$ Similarly we claim $\vec{m+1}\notin\Pcal$
cannot hold as in that case $\vec{m}\notin\Pcal_{\alpha}.$ This implies
$\vec{m}, \vec{m+1}\in\Pcal$. Notice $p_{\pi}^{-1}(\vec{j})=\{\vec{j}\}$
for $j\leq m-1$. We conclude $\Pcal=\phi_{m+1}$ using Lemma \ref{lemma:trivial}.
\item \label{case:pcal=00003D00003D00003D00003D00003D00003DAcal_j_m} We
will divide $\Pcal_{\pi}=\Acal_{j,m}$ into two cases: \\
 a. $\Pcal_{\pi}=\Acal_{j,m}$ for $j<m$. By Lemma \ref{lem:pcal_j,m+1 is contained}
$\Acal_{j,m+1}\subset\pcal$. We only need to consider $P\in\Acal_{j,m}$
$|P|=j$ and $m\in P$, for which $p_{\pi}^{-1}(P)=\{P,\hat{P},P_{m+1}\}$,
and show that $P_{m+1}\notin\Pcal$. Indeed select $k\notin P$ with
$k<m$. Let $\alpha=\gamma_{k,k+1}^{m+1}$, then $|(P_{m+1})_{\alpha}|=j+1$
which implies $(P_{m+1})_{\alpha}\notin\Pcal_{\pi}.$ Finally we conclude
$\Pcal=\Acal_{j,m+1}$ using Lemma \ref{lemma:trivial}. \\
 b. $\Pcal_{\pi}=\Acal_{m,m}$. By Lemma \ref{lem:pcal_j,m+1 is contained}
$\Acal_{m-1,m+1}\subset\Pcal$. We therefore need to determine which
elements of $H_{11}^{m+1}\cup\{\vec{m+1}\}$ belong to $\Pcal$. Recall
from article (\ref{ite:Acal}) of Theorem \ref{thm:Standard_Examples_are_Stable}
that $(\Acal_{m-1,m+1})_{\gamma}=\Acal_{m-1,m}$ for all $\gamma\in\Pi{m+1 \choose m}$,
and notice that for any $P\in H_{1,1}^{m+1}\cup\{\vec{m+1}\},$ $P_{\gamma}=\vec{m}$.
First assume $\vec{m+1}\in\pcal.$ Conclude $\Pcal=\Acal_{m-1,m+1}\cup\ncal\cup\{\vec{m+1}\}$,
for $\ncal\subset H_{1,1}^{m+1}$, or $\Pcal=\Acal_{m-1,m+1}\cup\ncal'$
for $\ncal'\subset H_{1,1}^{m}.$ If $|\ncal'|=m+1$, $\pcal=\Acal_{m,m+1}.$
If $|\ncal'|<m$, then there exists $i,j\in\vec{m+1}$, $i\neq j$
so that for all $P\in\ncal'$, $i,j\in P$. This implies $\vec{m}\notin(\Acal_{m-1,m+1}\cup\ncal')_{\gamma_{i,j}}$.
Conclude $\Pcal=\Acal_{m-1,m+1}\cup\ncal'$ for $\ncal'\subset H_{1,1}^{m}$
with $|\ncal'|=m$ or $\pcal=\Acal_{m,m+1}$ or $\Pcal=\Acal_{m-1,m+1}\cup\ncal\cup\{\vec{m+1}\}$,
for $\ncal\subset H_{1,1}^{m+1}$.
\item \label{case:Acal_j_m_cup_vec(m)}$\Pcal_{\pi}=\Acal_{j,m}\cup\{\vec{m}\}$
for $1\le j\le m-2$. $p_{\pi}^{-1}(\vec{m})=\{\vec{m+1},\vec{m},\hat{m}\}$.
We claim $\vec{m}\notin\Pcal.$ Indeed let $\alpha=\gamma_{1,m}$
and observe that $\vec{m}_{\alpha}=\vec{m-1}\notin\Pcal_{\pi}$. Similarly
let $\beta=\gamma_{1,m+1}$ and observe that $\hat{m}_{\alpha}=\vec{m-1}\notin\Pcal$.
We now continue as in case (\ref{case:pcal=00003D00003D00003D00003D00003D00003DAcal_j_m}a).
Conclude $\Pcal=\Acal_{j,m+1}\cup\{\vec{m+1}\}$.
\item \label{case:Acal_j_m_phi_m}$\Pcal_{\pi}=\Acal_{j,m}\cup\phi_{m}$
for $1\le j\le m-2$. We start by analyzing $p_{\pi}^{-1}(P)$ for
$P\in\Acal_{j,m}$. This is done as in case (\ref{case:pcal=00003D00003D00003D00003D00003D00003DAcal_j_m})
with the sole difference that for the case $P\in\Acal_{j,m},\,|P|=j,\, m\in P$
we choose $1\leq i<k<m$ so that $i,k\notin P$ and notice that for
$\alpha=\gamma_{i,k}$, we have $|(P_{m+1})_{\alpha}|=j+1$ and in
addition $i\notin(P_{m+1})_{\alpha}$ and $m\in(P_{m+1})_{\alpha}$
which implies $(P_{m+1})_{\alpha}\notin\Pcal_{\pi}$. Let now $P\in\phi_{m}$,
so that $j<|P|<m-1$. Notice $p_{\pi}^{-1}(A)=\{A\}$. For $P=\vec{m}$,
we continue as in case (\ref{case:pcal=00003D00003D00003D00003D00003D00003Dphi_m}).
Conclude $\Pcal=\Acal_{j,m+1}\cup\phi_{m+1}$.
\item \label{case:Acal_j_m^1}We will divide $\Pcal_{\pi}=\Acal_{j,m}^{1}$
into two cases:\\
 a. $\Pcal_{\pi}=\Acal_{j,m}^{1}$ for $j<m$. Similar to the
proof of article (\ref{case:pcal=00003D00003D00003D00003D00003D00003DAcal_j_m}a).
Conclude $\Pcal=\Acal_{j,m+1}^{1}$. \\
 b. $\Pcal_{\pi}=\Acal_{m,m}^{1}.$ Similar to the proof of article
(\ref{case:pcal=00003D00003D00003D00003D00003D00003DAcal_j_m}b).
Notice that $\Pcal=\Acal_{m-1,m+1}^{1}\cup\ncal'$ for $\emptyset\neq\ncal'\subsetneq H_{1,2}^{m+1}$
is ruled out because in such a case $|\ncal'|<m$. Conclude $\Pcal=\Acal_{m-1,m+1}^{1}\cup\ncal\cup\{\vec{m+1}\}$,
for $\ncal\subset H_{1,2}^{m+1}$ or $\Pcal=\Acal_{m,m+1}^{1}$.
\item \label{case:Acal^1_ncal_vec(m)}$\Pcal_{\pi}=\Acal_{m-2,m}^{1}\cup\ncal\cup\{\vec{m}\}$
where $\emptyset\neq\ncal\subsetneq H_{1,2}^{m}$. Let $e=min_{Q\in\ncal}e(Q)$.
A similar argument to case (\ref{enu:Acal_m-2,m_cup_ncal_cup_vec(m)})
yields $e\geq3$, $\ncal=H_{1,e}^{m+1}$ and $\pcal=\Acal_{m-2,m+1}^{1}\cup H_{2,e}^{m+1}\cup H_{1,e+1}^{m+1}\cup\{\vec{m+1\}}=\dcal_{2,e+1}^{m+1}.$
%
{}
\item \label{enu:acal_cup_ncal}$\Pcal_{\pi}=\Acal_{m-2,m}\cup\ncal$ for
$\ncal\subset H_{1,1}^{m}$, $\ncal\neq H_{1,2}^{m}$ with $|\ncal|=m-1$.
This implies $\ncal=H_{1,1}^{m}\setminus EH_{j,1}^{m}$ for some $2\leq j\leq m$,
so this case corresponds to Lemma \ref{lem:H is in pcal} $h=1,l=1$
and we conclude $\Pcal_{\pi}\cap H_{1,1}^{m}=H_{1,1}^{m}$ which is
a contradiction.

\begin{enumerate}
\item \label{enu:acal_ncal_vec(m)}$\Pcal_{\pi}=\Acal_{m-2,m}\cup\ncal\cup\{\vec{m}\}$
for $\emptyset\neq\ncal\subsetneq H_{1,1}^{m}$, $\ncal\neq H_{1,2}^{m}$.
Let $e=min_{Q\in\ncal}e(Q)$. By Lemma \ref{lem:H is in pcal} $\ncal=H_{1,e}^{m}$
and $H_{2,e}^{m+1}=\Pcal\cap H_{2,1}^{m+1}$. By assumption $\ncal\neq H_{1,1}^{m},\, H_{1,2}^{m}$
and therefore $e>2$. For $e\geq3$, we use Lemma \ref{lem:H^m+1_1_1_in_pcal}
to conclude $H_{1,e+1}^{m+1}=\Pcal\cap H_{1,1}^{m+1}$ and $\vec{m+1\in\pcal.}$
By Lemma \ref{lem:pcal_j,m+1 is contained} $\Acal_{m-2,m+1}\subset\pcal$.
Finally conclude $\pcal=\Acal_{m-2,m+1}\cup H_{2,e}^{m+1}\cup H_{1,e+1}^{m+1}\cup\{\vec{m+1\}}=\dcal_{2,e+1}^{m+1}\cup\acal_{m-2,m+1}.$
\end{enumerate}
\item $\Pcal_{\pi}=\dcal_{r,s}^{m}$, for $2\leq r<r+1<s<m$. Recall $\dcal_{r,s}^{m}=\Acal_{m-r-1,m}^{1}\cup\bigcup_{l'=1}^{r}H_{l',s-l'+1}^{m}\cup\{\vec{m}\}$.
By Lemma \ref{lem:pcal_j,m+1 is contained}, $\Acal_{m-r-1,m}^{1}\subset\Pcal$.
Apply Lemma \ref{lem:H is in pcal} $r$ times w.r.t. pairs $h=l'$
and $l=s-l'+1$ to conclude $H_{l'+1,s-l'+1}^{m+1}\cap\Pcal=H_{1,s-l'+1}^{m+1}$.
Finally conclude $\pcal=\dcal_{r+1,s+1}^{m+1}$. %
{}
\item $\pcal_{\pi}=\dcal_{r,s}^{m}\cup\Acal_{j,m}$ for $2\leq r<r+1<s<m$
and $1\leq j\leq m-r-1$. $P\in\acal_{j,m}$ is analyzed as in case
(\ref{enu:Acal^1_j_m_cup_Acal_j'}). $P\in\dcal_{r,s}^{m}$ is analyzed
as in case (\ref{enu:dcal^m_s_r}). Conclude $\pcal=\dcal_{r+1,s+1}^{m+1}\cup\Acal_{j,m+1}$.
\item $\Pcal_{\pi}=\Acal_{j,m}^{1}\cup\{\vec{m}\}$ for $1\le j\le m-2$.
Similar to case (\ref{case:Acal_j_m_cup_vec(m)}). Conclude $\Pcal=\Acal_{j,m+1}^{1}\cup\{\vec{m+1}\}$.
\item \label{case:Pcal=00003D00003D00003D00003D00003D00003D00003DAcaljm1_phi_m}$\Pcal_{\pi}=\Acal_{j,m}^{1}\cup\phi_{m}$
for $1\le j\le m-2$. Similar to case (\ref{case:Acal_j_m_phi_m}).
Conclude $\Pcal=\Acal_{j,m+1}^{1}\cup\phi_{m+1}$. %
{}
\item \label{enu:acal^1_cup_acal}$\Pcal_{\pi}=\Acal_{j,m}^{1}\cup\Acal_{j',m}$
for $2\le j\le m$ and $j'<j$. By Lemma \ref{lem:pcal_j,m+1 is contained}
$\Acal_{j,m+1}^{1}\cup\Acal_{j',m+1}\subset\pcal$. We treat two cases:
\\
a. $j<m$. Similar to case (\ref{case:Acal_1jm_Acal_j'm_vec(m)}).
Conclude $\pcal=\Acal_{j,m+1}^{1}\cup\Acal_{j',m+1}$. \\
b. $j=m$.
Similar to case (\ref{case:Acal_j_m^1}b). Conclude
$\Pcal=\Acal_{j',m+1}\cup\Acal_{m-1,m+1}^{1}\cup\ncal\cup\{\vec{m+1}\}$
with $\ncal\subset H_{1,2}^{m+1}$ or $\Pcal=\Acal_{j',m+1}\cup\Acal_{m,m+1}^{1}$.
\item \label{case:Acal_1jm_Acal_j'm_vec(m)}$\Pcal_{\pi}=\Acal_{j,m}^{1}\cup\Acal_{j',m}\cup\{\vec{m}\}$
for $2\le j\le m-2$ and $j'<j$. $P\in\Acal_{j,m}^{1}\cup\{\vec{m}\}$
is analyzed as in case (\ref{case:Acal_j_m_cup_vec(m)}). For $P\in\Acal_{j',m}$,
we only need to deal with $P\in\Acal_{j',m}^{m}$, $1\notin P$, $|P|=j'<m-2$.
This is done as in case (\ref{case:Acal_j_m_phi_m}). Conclude $\Pcal=\Acal_{j,m+1}^{1}\cup\Acal_{j',m+1}\cup\{\vec{m+1}\}$.
\item $\Pcal_{\pi}=\Acal_{j,m}^{1}\cup\Acal_{j',m}\cup\phi_{m}$ for $2\le j\le m-2$
and $j'<j$ . $P\in\Acal_{j,m}^{1}\cup\phi_{m}$ is analyzed as in
case (\ref{case:Acal_j_m_cup_vec(m)}). $P\in\Acal_{j',m}$ is analyzed
as in case (\ref{case:Acal_1jm_Acal_j'm_vec(m)}). Conclude $\Pcal=\Acal_{j,m+1}^{1}\cup\Acal_{j',m+1}\cup\phi_{m+1}$.
\item \label{enu:acal_cup_acal^1_cup_ncal_cup_vec(m)}$\pcal_{\pi}=\Acal_{j,m}\cup\Acal_{m-2,m}^{1}\cup\ncal\cup\{\vec{m}\}$
where $\emptyset\neq\ncal\subsetneq H_{1,2}^{m}$ and $1\le j<m-2$.
$P\in\Acal_{j,m}$ is analyzed as in case (\ref{case:Acal_1jm_Acal_j'm_vec(m)}).
$\Acal_{m-2,m}^{1}\cup\ncal\cup\{\vec{m}\}$ is analyzed as in case
(\ref{case:Acal^1_ncal_vec(m)}). Conclude $\pcal=\Acal_{j,m+1}\cup\Acal_{m-2,m+1}^{1}\cup H_{2,e}^{m+1}\cup H_{1,e+1}^{m+1}\cup\{\vec{m+1\}}=\Acal_{j,m}\cup\dcal_{2,e+1}^{m+1}$
for some $e\geq3$.
\end{enumerate}
\end{proof}

\subsection{Hereditary patterns}
\begin{defn}\label{def:hereditary patterns}
An $m$-stable $m$-pattern $\Pcal$ is said to be \textbf{hereditary}
if for every $m'>m$ there exist an $m'$-stable $m'$-pattern $\qcal$
so that for any $\gamma\in\Pi\binom{m'}{m}$ it holds that $\pcal=\qcal_{\gamma}$.

\noindent Denote $HSP_{n}(m)=\{\Pcal : \,\Pcal\textrm{ is a hereditary \ensuremath{m}-stable \ensuremath{m_{n}}-pattern}\}$.%
{}\end{defn}

\begin{thm}
\label{thm:the only hereditary}The following are the only hereditary
$m$-stable $m_{1}$-patterns for $m\geq3$:
\begin{enumerate}
\item $\{\vec{m}\}$ .
\item $\phi_{m}$.
\item For every $1\le j\le m$ the collection $\Acal_{j,m}$.
\item For every $1\le j\le m-2$ the collection $\Acal_{j,m}\cup\{\vec{m}\}$.
\item For every $1\le j\le m-2$ the collection $\Acal_{j,m}\cup\phi_{m}$.
\item For every $1\le j\le m$ the collection $\Acal_{j,m}^{1}$.
\item \label{enu:D1}For every $1\leq r<r+1<s<m$\textup{ the collection
$\dcal_{r,s}^{m}$}.
\item \label{enu:D2}For every $1\leq r<r+1<s<m$\textup{ }and $1\leq j\leq m-r-1$,
$\dcal_{r,s}^{m}\cup\Acal_{j,m}$.
\item For every $1\le j\le m-2$ the collection $\Acal_{j,m}^{1}\cup\{\vec{m}\}.$
\item For every $1\le j\le m-2$ the collection $\Acal_{j,m}^{1}\cup\phi_{m}$.
\item For every $2\le j\le m$ and $j'<j$ the collection $\Acal_{j,m}^{1}\cup
\Acal_{j',m}$.
\item For every $2\le j\le m-2$ and $j'<j$ the collection $\Acal_{j,m}^{1}\cup\Acal_{j',m}\cup\{\vec{m}\}$.
\item For every $2\le j\le m-2$ and $j'<j$ the collection $\Acal_{j,m}^{1}\cup\Acal_{j',m}\cup\phi_{m}$.
\end{enumerate}
\end{thm}
\begin{proof}
This follows from the proof of Theorem \ref{thm: standard are the only stable}. \end{proof}

Note that this list is the list of standard patterns (Definition \ref{standard})
with the items (7),(8),(9), and (17) removed (notice however that in some cases the allowed indices slightly differ).


\begin{lem}
\label{lem: second representation of D}
Let $2\leq r<r+1<s<m$, then
$$
\textup{$\dcal_{r,s}^{m}=\{\{\vec{l}\cup\{j_{1},\ldots,j_{m-s}\}\}_{1\leq l\leq m,\,1\leq j_{1}\leq\ldots\leq j_{m-s}}\cup\Acal_{m-r-1,m}^{1}$.}
$$
\end{lem}
\begin{proof}
Recall $\dcal_{r,s}^{m}\triangleq\Acal_{m-r-1,m}^{1}\cup\bigcup_{h=1}^{r}H_{h,s-h+1}^{m}\cup\{\vec{m}\}.$
Notice that for $1\leq h\leq r$, $H_{h,s-h+1}^{m}=\vec{\{s-h}\cup\{j_{1},\ldots,j_{m-s}\}\}_{s-h+1\leq j_{1}<j_{2}<\ldots<j_{m-s}}$
so clearly the left hand side is contained in the right hand side.
Fix $1\leq l\leq m$ and $1\leq j_{1}\leq\ldots\leq j_{m-s}$. We
will show $F\triangleq\vec{l}\cup\{j_{1},\ldots,j_{m-s}\}\in\dcal_{r,s}^{m}$.
If $F\in\Acal_{m-r-1,m}^{1}$, we are done, so assume $F\notin\Acal_{m-r-1,m}^{1}$.
This implies the number of holes of $F$, which we will denote by
$h,$ is less or equal $r$. We assume w.l.o.g. $h\geq1$. Let $e$
be the first hole of $F$. We will show $e\geq s-h+1$ which will
imply $F\in H_{h,s-h+1}^{m}$. Assume for a contradiction that $e<s-h+1$.
This implies $l<e\leq s-h$ and $|F|<s-h+m-s=m-h$. However as $F$
has exactly $h$ holes $|F|=m-h$ and we have the desired contradiction.
\end{proof}

\begin{defn}
Let $\pcal\in HSP_{n}(m)$. We say that $\pcal$ is \textbf{permutation
stable} if $\sigma\pcal\in HSP_{n}(m)$ for some $\sigma\in S_{m}$
implies $\sigma\pcal=\pcal$.
\end{defn}


\begin{thm}
\label{thm:hereditary is permutation stable}The hereditary $m$-stable
$m_{1}$-patterns for $m\geq3$ are permutation stable.\end{thm}
\begin{proof}
This is proven case by case using the list of Theorem \ref{thm:the only hereditary}.
The only slightly non-trivial cases are articles (\ref{enu:D1}) and
(\ref{enu:D2}) where one uses the representation of Lemma \ref{lem: second representation of D}. \end{proof}

\begin{defn}\label{usl}
Let $\pcal\in SP_{n}(m)$. We say that $\pcal$ has \textbf{unique
stable lifts (usl)} if for every $m'>m$ there exists a unique $\qcal\in SP_{n}(m')$
so that for any $\gamma\in\Pi\binom{m'}{m}$ it holds that $\pcal=\qcal_{\gamma}$. \end{defn}

\begin{rem}
If $\pcal\in SP_{n}(m)$ has usl then $\pcal\in HSP_{n}(m).$\end{rem}

\begin{thm}
\label{thm:usl patterns}The following $m_{1}$-patterns $(m\geq3)$
have unique stable lifts:
\begin{enumerate}
\item $\{\vec{m}\}$ .
\item $\phi_{m}$.
\item For every $1\le j\le m-1$ the collection $\Acal_{j,m}$.
\item For every $1\le j\le m-2$ the collection $\Acal_{j,m}\cup\{\vec{m}\}$.
\item For every $1\le j\le m-2$ the collection $\Acal_{j,m}\cup\phi_{m}$.
\item For every $1\le j\le m-1$ the collection $\Acal_{j,m}^{1}$.
\item \label{usl: dcal}For every $1\leq r<r+1<s<m$\textup{ the collection
$\dcal_{r,s}^{m}$}.
\item \label{usl: dcal_acal}For every $1\leq r<r+1<s<m$\textup{ }and $1\leq j\leq m-r-1$,
$\dcal_{r,s}^{m}\cup\Acal_{j,m}$.
\item For every $1\le j\le m-2$ the collection $\Acal_{j,m}^{1}\cup\{\vec{m}\}.$
\item For every $1\le j\le m-2$ the collection $\Acal_{j,m}^{1}\cup\phi_{m}$.
\item For every $2\le j\le m-1$ and $j'<j$ the collection $\Acal_{j,m}^{1}\cup\Acal_{j'}$.
\item For every $2\le j\le m-2$ and $j'<j$ the collection $\Acal_{j,m}^{1}\cup\Acal_{j',m}\cup\{\vec{m}\}$.
\item For every $2\le j\le m-2$ and $j'<j$ the collection $\Acal_{j,m}^{1}\cup\Acal_{j',m}\cup\phi_{m}$.
\end{enumerate}
\end{thm}

\begin{proof}
The proof follows easily from the proof of Theorem \ref{thm: standard are the only stable}.\end{proof}

Note that this list is the list of Theorem \ref{thm:the only hereditary}
with the items (3),(6) and (11) for the case $j=m$ removed.
These cases do not have usl as the following lemma shows.

\begin{lem}
\label{lem:non usl cases}Let $m\geq3.$ The following holds:
\begin{enumerate}
\item
$p_{\pi}^{-1}(\Acal_{m,m})\cap HSP_{1}(m+1)= \\
\{\Acal_{m,m+1},\Acal_{m+1,m+1},\Acal_{m-1,m+1}\cup\Acal_{m,m+1}^{1},\Acal_{m-1,m+1}\cup\{\vec{m+1}\}\}\\
\cup\{\dcal_{1,l}^{m+1} \cup\acal_{m-1,m+1}\}_{l\in\{2,\ldots,m+1\}}.$
\item
$p_{\pi}^{-1}(\Acal_{m,m}^{1})\cap HSP_{1}(m+1)= \\
\{\Acal_{m,m+1}^{1},\Acal_{m+1,m+1}^{1},\Acal_{m-1,m+1}^{1}\cup\{\vec{m+1}\}\}\cup\{\dcal_{1,l}^{m+1}\}_{l\in\{3,\ldots,m+1\}}$.
\item
For
$1\leq j\leq m-2$,
$p_{\pi}^{-1}(\Acal_{m,m}^{1}\cup\acal_{j,m}) \cap HSP_{1}(m+1) =\\
\{\Acal_{m,m+1}^{1}\cup\Acal_{j,m+1},\Acal_{m-1,m+1}^{1}\cup\Acal_{j,m+1}\cup\{\vec{m+1}\},\Acal_{m+1,m+1}^{1}\cup\Acal_{j,m+1}\}\\
\cup\{\dcal_{1,l}^{m+1}\cup\Acal_{j,m+1}\}_{l\in\{3,\ldots,m+1\}}$.
\end{enumerate}
\end{lem}

\begin{proof}
$ $
\begin{enumerate}
\item Let $\pcal\in p_{\pi}^{-1}(\Acal_{m,m})\cap HSP_{1}(m+1).$ According
to article (\ref{enu:Acal_j_m}) in the proof of Theorem \ref{thm: standard are the only stable},
$\Pcal_{\pi}=\Acal_{m,m}$ implies $\Pcal=\Acal_{m-1,m+1}\cup\ncal'$
for $\ncal'\subset H_{1,1}^{m}$ with $|\ncal'|=m$ or $\pcal=\Acal_{m,m+1}$
or $\Pcal=\Acal_{m-1,m+1}\cup\ncal\cup\{\vec{m+1}\}$, for $\ncal\subset H_{1,1}^{m+1}$.
By article (\ref{enu:acal_cup_ncal}), if $\ncal\neq H_{2,1}^{m}$
in the proof of the Theorem \ref{thm: standard are the only stable},
$\Acal_{m-1,m+1}\cup\ncal'\notin HSP_{1}(m+1)$. If $\ncal=H_{2,1}^{m}$,
we have $\pcal=\Acal_{m-1,m+1}\cup\Acal_{m,m+1}^{1}$. By article
(\ref{enu:acal_ncal_vec(m)}) in the proof of Theorem \ref{thm: standard are the only stable},
$\Acal_{m-1,m+1}\cup\ncal\cup\{\vec{m+1}\}\in HSP_{1}(m+1)$ implies
for $\ncal\neq\emptyset$, $\ncal=H_{1,l}^{m+1}$ for some
$l\in\{1,2,\ldots,m+1\}$.
$l=1$ corresponds to $\pcal=\acal_{m+1,m+1}$ and $l\geq2$ corresponds
to $\pcal=\dcal_{1,l}^{m+1}\cup\acal_{m-1,m+1}$. If $\ncal=\emptyset$,
$\pcal=\Acal_{m-1,m+1}\cup\{\vec{m+1}\}$.
\item Let $\pcal\in p_{\pi}^{-1}(\Acal_{m,m}^{1})\cap HSP_{1}(m+1).$ According
to article (\ref{case:Acal_j_m^1}) in the proof of Theorem \ref{thm: standard are the only stable},
$\Pcal_{\pi}=\Acal_{m,m}^{1}$ implies $\pcal=\Acal_{m,m+1}^{1}$
or $\Pcal=\Acal_{m-1,m+1}^{1}\cup\ncal\cup\{\vec{m+1}\}$, for $\ncal\subset H_{1,2}^{m+1}$.
By article (\ref{case:Acal^1_ncal_vec(m)}) in the proof of Theorem
\ref{thm: standard are the only stable}, $\Acal_{m-1,m+1}^{1}\cup\ncal\cup\{\vec{m+1}\}\in HSP_{1}(m+1)$
for $\ncal\neq\emptyset$implies $\ncal=H_{1,l}^{m+1}$ for some
$l\in\{2,\ldots,m+1\}$.
$l=2$ corresponds to $\pcal=\acal_{m+1,m+1}^{1}$ and $l\geq3$ corresponds
to $\pcal=\dcal_{1,l}^{m+1}$. If $\ncal=\emptyset$, $\pcal=\Acal_{m-1,m+1}^{1}\cup\{\vec{m+1}\}$.
\item
Let $\pcal\in p_{\pi}^{-1}(\Acal_{m,m}^{1}\cup\acal_{j,m})\cap HSP_{1}(m+1)$
for some
$1\leq j\leq m-2$.
According to article (\ref{enu:acal^1_cup_acal})
in the proof of Theorem \ref{thm: standard are the only stable},
$\Pcal_{\pi}=\Acal_{m,m}^{1}\cup\acal_{j,m}$ implies $\Pcal=\Acal_{j,m+1}\cup\Acal_{m-1,m+1}^{1}\cup\ncal\cup\{\vec{m+1}\}$
with $\ncal\subset H_{1,2}^{m+1}$ or $\Pcal=\Acal_{j,m+1}\cup\Acal_{m,m+1}^{1}$.
By article (\ref{enu:acal_cup_acal^1_cup_ncal_cup_vec(m)}) in the
proof of Theorem \ref{thm: standard are the only stable}, $\Acal_{j,m+1}\cup\Acal_{m-1,m+1}^{1}\cup\ncal\cup\{\vec{m+1}\}\in HSP_{1}(m+1)$
for $\ncal\neq\emptyset$ implies $\ncal=H_{1,l}^{m+1}$ for some
$l\in\{2,\ldots,m+1\}$. $l=2$ corresponds to $\pcal=\Acal_{j,m+1}\cup\Acal_{m+1,m+1}^{1}$
and $l\geq3$ corresponds to $\pcal=\dcal_{1,l}^{m+1}\cup\Acal_{j,m+1}$.
If $\ncal=\emptyset$, $\pcal=\Acal_{m-1,m+1}^{1}\cup\Acal_{j,m+1}\cup\{\vec{m+1}\}$.
\end{enumerate}
\end{proof}

\section{An application of the dual Ramsey Theorem to stable patterns}

The tool which enables the application of the
combinatorial results
of the previous section to hyperspace actions is the dual Ramsey Theorem.

\subsection{Ramsey Theorems}\label{Ramsey}

We denote by $\tilde{\Pi}\binom{s}{k}$ the collection of unordered
partitions of $\{1,\ldots,s\}$ into $k$ nonempty sets. Notice there
is a natural bijection $\nu:\tilde{\Pi}\binom{s}{k}\leftrightarrow\Pi\binom{s}{k}.$
\begin{thm}
\label{thm:Dual_Ramsey}{[}The dual Ramsey Theorem{]} Given positive
integers $k,m,r$ there exists a positive integer $N=DR(k,m,r)$ with
the following property: for any coloring of $\tilde{\Pi}\binom{N}{k}$
by $r$ colors there exists a partition $\alpha=\{A_{1},A_{2},\dots,A_{m}\}\in\tilde{\Pi}\binom{N}{m}$
of $N$ into $m$ non-empty sets such that all the partitions of $N$
into $k$ non-empty sets (i.e. elements of $\tilde{\Pi}\binom{N}{k}$)
whose atoms are measurable with respect to $\alpha$ (i.e. each equivalence
class is a union of elements of $\alpha$) have the same color. \end{thm}
\begin{proof}
This is Corollary 10 of \cite{GR71}.%
{}\end{proof}
\begin{thm}
{[}The strong dual Ramsey Theorem{]} Given positive integers $m,r_{2},\allowbreak\dots,r_{m}$
there exists a positive integer $N=SDR(m;r_{2},\dots,r_{m-1})$ with
the following property: for any colorings of $\tilde{\Pi}\binom{N}{j}$
by $r_{j}$ colors, $2\le j\le m-1$, there exists a partition $\alpha=(A_{1},A_{2},\dots,A_{m})\in\tilde{\Pi}\binom{N}{m}$
of $N$ into $m$ non-empty sets such that for any $2\le j\le m-1$,
any two partitions of $N$ into $j$ non-empty sets whose atoms are
measurable with respect to $\al$ (i.e. each equivalence class is
a union of elements of $\alpha$) have the same color. \end{thm}
\begin{proof}
{[}Using the dual Ramsey Theorem{]} Set $n_{2}=DR(2,m;r_{2})$, $n_{3}=DR(3,n_{2};r_{3})$
and, by recursion $n_{j+1}=DR(j+1,n_{j};r_{j+1})$ for $2\le j<m-2$.
It is now easy to check that $n=n_{m-1}$ satisfies our claim. (As
a demonstration set $n_{2}=DR(2,m;r_{2}),\, N=n_{3}=DR(3,n_{2};r_{3})$.
Start with a partition $\al_{3}=\{B_{1},\dots,B_{{n_{2}}}\}$ of $N$
which is good for all partitions of $N$ into $j=3$ atoms and the
$r_{3}$ colors. Next choose a partition $\al_{2}=\{A_{1},\dots,A_{m}\}$
of $n$ which is $\al_{3}$-measurable, and which is good for all
partitions of $N$ into $j=2$ atoms and the $r_{2}$ colors. It now
follows that $\al:=\al_{2}=\{A_{1},A_{2},\dots,A_{m}\}$ has the required
property: all $2$-partitions of $N$ which are $\al$-measurable
are monochromatic and all $3$-partitions of $N$ which are $\al$-measurable
are monochromatic.) \end{proof}
\begin{cor}
\label{stable-cor} Let $m,n$ be positive integers. For any number $N\geq SDR(m;r_{2},\dots,r_{m-1})$,
with $r_{k}=2^{(2^{k}-1)^{n}}-1$, for any $N_{n}$-pattern $\Pcal$, there
exists a partition $\alpha\in\Pi\binom{N}{m}$ such that the $m_{n}$-pattern
$\Pcal_{\alpha}$ is \textup{an $m$-stable $m_{n}$-pattern. }\end{cor}
\begin{proof}
{[}Using the strong dual Ramsey Theorem{]} We define for $2\leq j\leq m-1$ the colorings $f_j:\binom{N}{k}\rightarrow C_{n}(k)$
by $\alpha\mapsto\Pcal_{\alpha}$ in $r_k=|C_{n}(k)|=2^{(2^{k}-1)^{n}}-1$ colors.
By the strong dual Ramsey Theorem applied to $f\circ\nu$ there exists
a partition $\alpha=(A_{1},A_{2},\dots,A_{m})\in\Pi\binom{N}{m}$
of $N$ into $m$ non-empty sets such that for any $(m-2)$-tuple
of naturally ordered partitions of $N$ into $k=2,\ldots,m-1$ non-empty
sets (i.e. elements of $\Pi\binom{N}{k}$) respectively whose atoms
are measurable with respect to $\alpha$ (i.e. each equivalence class
is a union of elements of $\alpha$) have the same color. Let $\beta_{1}$
and $\beta_{2}$ be two naturally ordered partitions of $m$ into
$k$ elements for some $2\leq k\leq m-1$, then as the amalgamated
partitions $\alpha_{\beta_{1}},\alpha_{\beta_{2}}$ are measurable
with respect to $\alpha$ and naturally ordered, $(\Pcal_{\alpha})_{\beta_{1}}=\Pcal_{\alpha_{\beta_{1}}}=\Pcal_{\alpha_{\beta_{2}}}=(\Pcal_{\alpha})_{\beta_{2}}$
as desired.
\end{proof}

\section{The minimal subspaces of $Exp(Exp(X))$}

\subsection{Clopen partitions}

Let $X$ be a Hausdorff zero-dimensional compact h-homogeneous space.
Denote by $\D$ $(\tilde{\D})$ the directed set (semilattice) consisting
of all finite ordered (unordered) clopen partitions of $X$. We denote
the members of $\D$ $(\tilde{\D})$ by $\alpha=(A_{1},{A}_{2},\ldots,{A}_{m})$
$(\tilde{\alpha}=\{A_{1},A_{2}, \ldots,A_{m}\})$.
The relation is given by refinement: $\alpha\preceq{\beta}$
$(\tilde{\alpha}\preceq\tilde{\beta})$ iff for any ${B}\in \beta$
$({B}\in\tilde{\beta})$, there is $A \in \alpha $ $(A \in\tilde{\alpha})$
so that ${B}\subset{A}$. The join (least upper bound) of
$\alpha$ and $\beta$, $\alpha\vee\beta=\{A\cap{B} : \: A\in \alpha,\, B \in\beta\}$,
where the ordering of indices is given by the lexicographical order
on the indices of $\alpha$ and $\beta$ $(\tilde{\alpha}\vee\tilde{\beta}
=\{A \cap B : \:A\in\tilde{\alpha},\,B \in\tilde{\beta}\})$.
It is convenient to introduce the notations $\D_{k}=\{\alpha \in\D : \:|\alpha|=k\}$
and $\tilde{\D_{k}}=\{\alpha\in\D : \:|\tilde{\alpha}|=k\}$. There
is a natural $G$-action on $\D$ $(\tilde{\D})$ given by $g({A}_{1},{A}_{2},\ldots,{A}_{m})=(g({A}_{1}),g({A}_{2}),\ldots,g({A}_{m}))$
($g\{{A}_{1},{A}_{2},\ldots,{A}_{m}\}=\{g({A}_{1}),g({A}_{2}),\ldots,g({A}_{m})\}$).
Let $S_{k}$ denote the group of permutations of $\{1,\ldots,k\}$.
$S_{k}$ acts naturally on $\D_{k}$ by $\sigma({B}_{1},{B}_{2},\ldots,{B}_{k})=({B}_{\sigma(1)},{B}_{\sigma(2)},\ldots,{B}_{\sigma(k)})$.
This action
commutes
with the action of $G$, i.e. $\sigma g{\beta}=g\sigma{\beta}$
for any $\sigma\in S_{k}$ and $g\in G$. Notice one can identify
$\tilde{\D_{k}}=\D_{k}/S_{k}$.

For $\alpha=({A}_{1},{A}_{2},\ldots,{A}_{m})\in\D$
and $\gamma=(C_{1},\ldots,C_{k})\in\Pi\binom{m}{k}$ define the \textbf{amalgamated
clopen cover }$\alpha_{\gamma}=({G}_{1},{G}_{1},\ldots,{G}_{k})$,
where ${G}_{j}=\bigcup_{i\in C_{j}}{A}_{i}$. Notice that
$(\alpha_{\gamma})_{\beta}={\alpha}_{\gamma_{\beta}}.$

\subsection{Partition Homogeneity}

The following definition was introduced in \cite{GG12}:
\begin{defn}
A zero-dimensional Hausdorff space $X$ is called \textbf{partition-homogeneous}
if for every two finite ordered clopen partitions of the same cardinality,
$\alpha, \beta\in\D_{m}$, $\alpha=(A_{1},A_{2},\ldots,A_{m})$,
$\beta=(B_{1},B_{2},\ldots,B_{m})$ there is
$h\in Homeo(X)$ such that
$hA_{i}=B_{i}$, $i=1,\ldots,k$.
\end{defn}
In \cite{GG12} we proved:
\begin{prop}
\label{pro:h_hom_eq_part_hom}Let $X$ be a zero-dimensional compact
Hausdorff space. If $X$ is h-homogeneous
then $X$ is partition-homogeneous.
\end{prop}

\subsection{Signatures and Induced patterns }

For every $n\in\mathbb{N}$ let $E_{n}=Exp(Exp(X))^{n}$ equipped
with the Vietoris topology.
\begin{defn}
Let $\xi\in E_{n}$ and $\alpha \in\D$ with $m$ elements $\alpha=(A_{1},A_{2},\ldots,{A}_{m})$.
Define the
\textbf{$(\alpha,\xi)$-induced $m_{n}$-pattern}
$\Pcal_{(\alpha,\xi)}=\{P_{(\alpha,F)} : F\in\xi\}$, where
$P_{(\alpha,F)}=\{(j_{1},j_{2},\ldots,j_{n}) : \, A_{j_{1}}\times A_{j_{2}}\times\cdots\times A{}_{j_{n}}\cap F\ne\emptyset\}$.
Inversely to an $m_{n}$-pattern $\Pcal$ and clopen partition $\alpha$
one associates $\xi_{(\alpha,\Pcal)}=\{F_{(\alpha,P)} : \, P\in\Pcal\}\in E_{n}$,
where \[
F_{(\alpha,P)}=\cup_{(j_{1},j_{2},\ldots,j_{n})\in P}A_{j_{1}}\times A_{j_{2}}\times\cdots\times  A{}_{j_{n}}\]

Denote $\xi_{\alpha}=\xi_{(\alpha,\Pcal_{(\alpha,\xi)})}$.
It is easy to see:.

\[
\xi_{\alpha}=\{F_{\alpha} : \, F\in\xi\}\]

where \[
F_{\alpha}=\cup_{A_{j_{1}}\times A_{j_{2}}\times\cdots\times A{}_{j_{n}}\cap F\neq\emptyset}A_{j_{1}}\times A_{j_{2}}\times\cdots\times A{}_{j_{n}}\]

The \textbf{signature }of $\xi$ is defined to be the net in
$E_n$,
$sig(\xi)=(\xi_{\alpha})_{\alpha \in\D}$. \end{defn}

\begin{lem}
\label{lem:Pattern_Transformation}Let $\xi\in E_{n}$, $\alpha \in\D_{m}$,
$\sigma\in S_{m}$ and $g\in G$ then \end{lem}
\begin{enumerate}
\item $\Pcal_{(\alpha,\xi)}=\Pcal_{(g\alpha,g\xi)}$.
\item $\sigma^{-1}\Pcal_{(\alpha,\xi)}=\Pcal_{(\sigma\alpha,\xi)}$\end{enumerate}
\begin{proof}
Let $\alpha=(A_{1},A_{2}\ldots,A_{m})$.\end{proof}
\begin{enumerate}
\item Notice that for any $F\in\xi$, $F\cap A_{j_{1}} \times A_{j_{2}}\times\cdots\times A{}_{j_{n}}\neq\emptyset\Leftrightarrow g(F)\cap gA_{j_{1}}\times gA_{j_{2}}\times\cdots\times gA{}_{j_{n}}\neq\emptyset$.
\item Notice that for any $F\in\xi$, $F\cap{A}_{\sigma(j_{1})}\times A_{j_{2}}\times\cdots\times A{}_{\sigma(j_{n})}\neq\emptyset\Leftrightarrow(j_{1},j_{2},\ldots,j_{n})\in\sigma^{-1}P_{({\alpha},F)}$.
\end{enumerate}

\subsection{The topology of $E_{n}$}

Recall that for a topological space $K$, a basis for the Vietoris topology on
$Exp(K)$ is given by the collection of subsets of the form
$$
\langle V_1, V_2, \dots, V_k \rangle =
\{F \in Exp(K) : \ F \subset \cup_{j=1}^k V_j,\ {\text{\rm and}}\
F \cap V_j \ne \emptyset\ \ \forall \ 1 \le j \le k\},
$$
where $V_1,\dots,V_k$ are open subsets of $K$.

\begin{lem}\label{basis}
For $\alpha=(A_{1},A_{2},\ldots,A_{m})$ let
%
\[
UE_{n}(\alpha)=
\{\Pi_{i=1}^{n}
\langle A_{j_{1}^{i}},A_{j_{2}^{i}},\ldots,A_{j_{k_{i}}^{i}} \rangle :
\ \forall i,  \forall r,\, k_{i},  j_{r}^{i}\in\vec{{m}},(\forall s,r\neq r')\, j_{r}^{s}\neq j_{r'}^{s}\}\]
and
\[
\mathcal{B}{}_{E_{n}}(\alpha)=\{\langle\Ucal_{1},\dots,\Ucal_{l}\rangle : \ \forall i,\,\Ucal_{i}\in UE_{n}(\alpha)\}.\]
Define $\mathcal{B}{}_{E_{n}}=\bigcup_{\alpha\in\D}\mathcal{B}{}_{E_{n}}({\alpha})$.
Then $\mathcal{B}{}_{E_{n}}$ is a basis for $E_{n}.$\end{lem}
\begin{proof}
A basis for the Vietoris topology of $E_{n}$ is given by $\mathfrak{U}=\langle\Ucal_{1},\dots,\Ucal_{l}\rangle$,
where $\Ucal_{i}$ belong to a fixed basis in $(Exp(X))^{n}.$\end{proof}
\begin{rem}
\label{rem:disjointness of UE(alpha)}Notice that
distinct members of $UE_{n}(\alpha)$
are disjoint and that for each $F=(F_{1,}F_{2},\ldots,F_{n})\in(ExpX)^{n}$,
there
exists
a unique member $\Ucal\in UE_{n}(\alpha),$ so that
$F\in \Ucal$.
Denote $\Ucal=UE_{n}({\alpha})[F]$.
\end{rem}

\begin{defn}\label{def:U(xi,alpha)}
For $\alpha=(A_1,\ldots,A_m) \in \mathcal{D}$ and
$\pcal \in C_n(m)$
with $\pcal =\{P_s=(P^1_s, P_s^2, \cdots, P_s^n):
s=1,2,\dots,r\}$ and $P^i_s = \{j^i_{s,1},j^i_{s,2},\dots, j^i_{s,k(i,s)}\}$
let
\begin{align*}
UE_n(\alpha)[\pcal] = &
\langle
\Pi_{i=1}^{n}\langle {A}_{j^i_{1,1}}, {A}_{j^i_{1,2}},\dots, {A}_{j^i_{1,k(i,1)}}
\rangle,
\Pi_{i=1}^{n}\langle {A}_{j^i_{2,1}}, {A}_{j^i_{2,2}},\dots, {A}_{j^i_{2,k(i,2)}}
\rangle,\\
&  \dots,
\Pi_{i=1}^{n}\langle {A}_{j^i_{r,1}}, {A}_{j^i_{r,2}},\dots, {A}_{j^i_{r,k(i,r)}}
\rangle
\rangle.
\end{align*}


Define $U(\xi,\alpha)=
UE_n(\alpha)[\pcal_{(\alpha,\xi)}]$.
\end{defn}

\begin{prop}
\label{pro:closeness_equivalences}Let $\alpha\in\dcal$. Let
$\xi,\xi'\in E_{n}$. The following statements are equivalent:\end{prop}
\begin{enumerate}
\item $U(\xi,\alpha)=U(\xi',{\alpha})$.
\item $\xi\in U(\xi',\alpha)$ or $\xi'\in U(\xi,\alpha)$.
\item $UE(\alpha)[\pcal_{(\alpha,\xi)}]=UE(\alpha)[\pcal_{(\alpha,\xi')}]$.
\item $\pcal_{(\alpha,\xi)}=\pcal_{(\alpha,\xi')}$.
\item $\xi_{\alpha}=\xi'_{\alpha}$.\end{enumerate}
\begin{proof}
The equivalences $(1)\Leftrightarrow(2)\Leftrightarrow(3)\Leftrightarrow(4)$
follow from Definition \ref{def:U(xi,alpha)}.
Recall $\xi_{\alpha}=\xi_{(\alpha,\Pcal_{(\alpha,\xi)})}$
which implies $(4)\Rightarrow(5)$. The direction $(4)\Leftarrow(5)$
is trivial.\end{proof}
\begin{lem}
\label{lem:sig(xi)=00003D00003Dxi}Let $\xi\in E_{n}$, then $\lim sig(\xi)=\xi$
and $\{\xi\}=\bigcap_{\alpha\in\D}U(\xi,{\alpha})$.\end{lem}
\begin{proof}
We will show that for any $\alpha\in\D$, $\xi_{{\beta}}\in U(\xi,\alpha)$
for all $\beta\succeq \alpha$. This implies that the limit
of the net $(\xi_{\alpha})_{\alpha\in\D}$ is $\xi$.
Fix $\beta=({{B}_{0},}{{B}_{2},}\ldots,{{B}}_{m-1})\succeq \alpha$.
Acording to Proposition \ref{pro:closeness_equivalences} it is enough
to show $(\xi_{{\mathbf{\beta}}})_{ \alpha} =\xi_{{\mathbf \alpha}}$.
However as $ \beta \succeq \alpha $ this is trivial. In order
to prove the second statement of the lemma notice that trivially $\xi\in\bigcap_{\alpha\in\D}U(\xi, \alpha )$.
If $\xi'\in\bigcap_{ \alpha \in\D}U(\xi,\alpha)$ then $\lim sig(\xi')=\lim sig(\xi)$
which implies by the first statement of the lemma $\xi=\xi'$.
\end{proof}

\subsection{Hereditary stable signatures}

\noindent Recall that $HSP_{n}(m)$, the collection of hereditary $m$-stable $m_{n}$-patterns was introduced in Definition \ref{def:hereditary patterns}. Notice there is a natural action of $S_{m}$ on $C_{n}(m)$.
\begin{defn}
Let $\xi\in E_{n}$. The signature $sig(\xi)$ is called \textbf{hereditary
stable} if for all $\alpha\in\D$, there exists $\sigma\in S_{|\alpha|}$
so that $\pcal_{(\sigma(\alpha),\xi)}\in HSP_{n}(m)$. \end{defn}
\begin{lem}
\label{lem:stability_permutation_conserved}Let $\xi\in E_{n}$. Let
$\beta\in\D$ with $|\beta|=m$. Assume \textup{$\Pcal_{(\beta,\xi)}\in HSP_{n}(m)$.}
Let $\alpha \in\D$ with $ \alpha \preceq \beta$ and
$| \alpha |=p$, then there exists $\sigma\in S_{m}$ so that
$\pcal_{(\sigma( \alpha ),\xi)}\in HSP_{n}(p)$. \end{lem}
\begin{proof}
Denote $ \alpha =( A _{1},\dots, A _{p})$ and
$ \beta =(B_{1},\dots,B_{m})$. Define $\tilde{C_{i}}=\{j : \,B_{j}\subset A_{i}\}$,
$i=1,\ldots,p$. Let $\tilde{\gamma}=\{\tilde{C_{i}}\}_{i=1}^{p}$.
Let $\gamma=(C_{i})_{i=1}^{p}$ be a naturally ordered partition of
$\{1,\ldots,m\}$ into $p$ sets, so that there exists $\sigma\in S_{m}$
so that $C_{i}=\tilde{C}_{\sigma(i)}$, $i=1,\ldots p$. Denote $\nu=\beta_{\gamma}$.
It is easy to see that $\sigma({\nu})=\alpha$. Moreover
$\pcal_{({\nu},\xi)}=(\Pcal_{(\beta,\xi)})_{\gamma}$. The
last assertion implies $\pcal_{(\sigma(\alpha),\xi)}\in HSP_{n}(p)$. \end{proof}

Our next theorem is a crucial step that connects the combinatorial
condition of being hereditarily stable to the topological dynamical condition of
being minimal. Note that whereas our main result (Theorem \ref{main2} below)
handles the case $n=1$ only, this result applies to all $n\in \N$.
\begin{thm}
\label{thm:Minimal Space contains stable signature}Let $M\subset E_{n}$
be minimal, then for all $\xi'\in M$, $sig(\xi')$ is hereditary
stable. \end{thm}
\begin{proof}
We start by showing there is $\xi'\in M$ so that $sig(\xi')$ is
hereditary stable. Fix $\xi\in M$. Let $\alpha=(A_{1},\dots,A_{m})\in\D$.
Let $N=SDR(m+1;r_{2},\dots,r_{m})$, with $r_{k}=2^{2^{kn}}$, $k=2,\ldots,m$.
Let $\beta \in\D$ with $|\beta|=N$
(see subsection \ref{Ramsey}). By Corollary \ref{stable-cor}
there exists a partition $\tilde{\gamma}\in\Pi\binom{N}{m+1}$ such
that the $(m+1)$-pattern $(\Pcal_{(\beta,\xi)})_{\tilde{\gamma}}=\Pcal_{(\beta_{\tilde{\gamma}},\xi)}$
is $(m+1)$-stable. Fix $\delta\in\Pi\binom{m+1}{m}$. Denote $\gamma=\tilde{\gamma}_{\delta}$.
Conclude that the $m$-pattern $(\Pcal_{(\beta,\xi)})_{\gamma}=\Pcal_{(\beta_{\gamma},\xi)}$
is a hereditary $m$-stable $m$-pattern. Denote $\beta_{\gamma}=\{{A'}_{1},\dots,{A'}_{m}\}$.
Let $g\in G$ be any element $g=g_{\alpha}$ of $G$ with $g({A'}_{j})={A}{}_{j},\ j=1,\dots,m$.
By Lemma \ref{lem:Pattern_Transformation} $\Pcal_{(\beta_{\gamma},\xi)}=\Pcal_{(\alpha,g_{\alpha}\xi)}$.
In particular $\Pcal_{(\alpha,g_{\alpha}\xi)}\in HSP_{n}(m)$.
Let $\xi'$ be a limit point of the net $\{g_{\alpha}\xi: \alpha\in\D\}$.
By definition there is a directed set $S$ and a monotone cofinal
mapping $f:S\rightarrow\D$, so that $\xi'=\lim_{s\in S}g_{f(s)}\xi.$
Fix again some $\alpha\in\D.$ By definition there exists $s\in S$
so that for all $\tilde{s}\geq s$, $g_{f(\tilde{s})}\xi\in U(\xi',\alpha)$
(which implies $g_{f(\tilde{s})}\xi\in U(\xi',\sigma(\alpha))$
for any $\sigma\in S_{m}$) and in particular $\xi'_{\alpha}=(g_{f(\tilde{s})}\xi)_{\alpha}$.
As $f$ is cofinal there is $r\in S$ so that $f(r)\succeq \alpha \wedge f(s)$.
By construction $\Pcal_{(f(r),g_{f(r)}\xi)}$ is a hereditary $|f(r)|$-stable
$|f(r)|$-pattern. As $f(r)\geq\alpha$, conclude by Lemma \ref{lem:stability_permutation_conserved}
$\pcal_{(\sigma(\alpha),g_{f(r)}\xi)}$ for some $\sigma\in S_{m}$
is a hereditary $m$- stable $m$-pattern. By Proposition \ref{pro:closeness_equivalences},
$\pcal_{(\sigma(\alpha),g_{f(r)}\xi)}=\pcal_{(\sigma(\alpha),\xi')}$.
We conclude $sig(\xi')$ is hereditary stable. Let now $\xi''\in M$
and fix ${\alpha}\in\D$. Using minimality and Proposition \ref{pro:closeness_equivalences} there
is $g\in G$ so that $\pcal_{(\alpha,\xi'')}=\pcal_{(\alpha,g\xi')}$.
By Lemma \ref{lem:Pattern_Transformation} $\pcal_{(\alpha,g\xi')}=\pcal_{(g^{-1}\alpha,\xi')}.$
As $sig(\xi')$ is hereditary stable there is $\sigma\in S_{m}$ so
that $\pcal_{(\sigma g^{-1}\alpha,\xi')}\in HSP_{n}(m)$. By
Lemma \ref{lem:Pattern_Transformation} $\pcal_{(\sigma g^{-1}\alpha,\xi')}=\sigma^{-1}\pcal_{(g^{-1}{\alpha},\xi')}=\sigma^{-1}\pcal_{(\alpha,g\xi')}=\sigma^{-1}\pcal_{(\alpha,\xi'')}=\pcal_{(\sigma\alpha,\xi'')}$
and we conclude $sig(\xi'')$ is hereditary stable.
\end{proof}

\subsection{The main theorem}
\begin{defn}
We call a sequence of $n$-dimensional patterns $\mathbb{P}=\{\Pcal_{m}\}_{m\in\mathbb{N}}$
so that $\Pcal_{m}\in HSP_{n}(m)$ a \textbf{pattern-family}. A minimal
subspace $M\subset E_{n}$ is said to be \textbf{$\mathbb{P}$-associated}
if for every $\xi\in M$, $m\in\mathbb{N}$ and $\alpha\in\D_{m}$
there exists $\sigma\in S_{m}$ (depending on $\xi$) so that $\xi_{\alpha}=\sigma P_{m}$.
One easily verfies that for a given pattern family $\mathbb{P}$ there
is at most one minimal subspace
to which it is $\mathbb{P}$-associated.
If such a subspace exists it is denoted by
$M(\mathbb{P})$.
\end{defn}
From now on we assume $n=1$. For the following definition, recall Definition \ref{standard} and use Lemma \ref{lem: second representation of D}
to
understand articles (\ref{enu:P^Q}), (\ref{enu:P^qL}), (\ref{enu:P^qj})
and (\ref{enu:P^Lqj}).


\begin{defn}
\label{def:Standard spaces}
The pattern-families in the following list are called
\textbf{standard}.
Each of these is associated with a minimal subspaces of $E_{1}$
and we call these minimal subspaces
the \textbf{standard minimal spaces}.
\begin{enumerate}
\item
$M(\{\vec{m}\}_{m\in\mathbb{N}})=\{\{X\}\}$.
\item
$M(\{\phi_{m}\}_{m\in\mathbb{N}})=\Phi$.
\item
$M(\{\Acal_{j,m}\}_{m\in\mathbb{N}})=\{\{\{x_{1},x_{2},\ldots,x_{j}\}\}_{(x_{1},x_{2},\ldots,x_{j})\in X^{j}}\}$
($j\in\mathbb{N}$).
\item
$M(\{\Acal_{j,m}\cup\{\vec{m}\}\}_{m\in\mathbb{N}})=\{\{\{x_{1},x_{2},\ldots,x_{j}\},X\}_{(x_{1},x_{2},\ldots,x_{j})\in X^{j}}\}$
($j\in\mathbb{N}$).
\item
$M(\{\Acal_{j,m}\cup\phi_{m}\}_{m\in\mathbb{N}})=\{\{\{x_{1},x_{2},\ldots,x_{j}\},F\}_{(x_{1},x_{2},\ldots,x_{j})\in X^{j},F\in\xi}\}_{\xi\in\Phi}$
($j\in\mathbb{N}$).
\item
$M(\{\Acal_{j,m}^{1}\}_{m\in\mathbb{N}})=\{\{\{x_{1},x_{2},\ldots,x_{j}\}_{(x_{2},\ldots,x_{j})\in X^{j-1}}\}\}_{x_{1}\in X}$
($j\in\mathbb{N}$).
\item\label{enu:P^Q}
Let $q\geq1$. Let $\mathbb{P}^{q}=\{\Pcal_{m}^{q}\}_{m\in\mathbb{N}}$
be given by\[
\Pcal_{m}^{q}=\begin{cases}
\Acal_{m,m}^{1} & 1\leq m\leq q+2\\
\dcal_{m-q-3,m-q}^{m} & m>q+2\end{cases}\]
$M(\mathbb{P}^{q})=\{\{F\cup\{x_{1},x_{2},\ldots,x_{q}\}\}_{(x_{1},x_{2},\ldots,x_{q})\in X^{q},F\in\xi}\}{}_{\xi\in\Phi}$.
\item\label{enu:P^qL}
Let $l>q\geq1$. Let $\mathbb{P}^{q,l}=\{\Pcal_{m}^{q,l}\}_{m\in\mathbb{N}}$
be given by\[
\Pcal_{m}^{q,l}=\begin{cases}
\Acal_{m,m}^{1} & 1\leq m\leq l\\
\dcal_{m-l-1,m-q}^{m} & m>l\end{cases}\]
$M(\mathbb{P}^{q,l})=\{\{F\cup\{x_{1},x_{2},\ldots,x_{q}\},\{r(\xi),y_{2},\ldots,y_{l}\}\}_{(x_{1},x_{2},\ldots,x_{q})\in X^{q},
(y_{2},\ldots,y_{l})\in X^{l-1},F\in\xi}\}{}_{\xi\in\Phi}$.
\item \label{enu:P^qj}
Let $q,j\geq1$. Let $\mathbb{P}^{q,j}=\{\Pcal_{m}^{q,j}\}_{m\in\mathbb{N}}$
be given by\[
\Pcal_{m}^{q,j}=\begin{cases}
\Acal_{j,m}\cup\Acal_{m,m}^{1} & 1\leq m\leq q+2\\
\Acal_{j,m}\cup\dcal_{m-l-1,m-q}^{m} & m>q+2\end{cases}\]
$M(\mathbb{P}^{q,l})=\{\{F\cup\{x_{1},x_{2},\ldots,x_{q}\},\{z_{1},z_{2},\ldots,z_{j}\}\}_{(x_{1},x_{2},\ldots,x_{q})\in X^{q},(z_{1},z_{2},\ldots,z_{j})\in X^{j},F\in\xi}\}{}_{\xi\in\Phi}$.
\item \label{enu:P^Lqj}
Let $l>q\geq1$ and $1\leq j<l$. Let $\mathbb{P}^{q,l,j}=\{\Pcal_{m}^{q,l,j}\}_{m\in\mathbb{N}}$
be given by\[
\Pcal_{m}^{q,l,j}=\begin{cases}
\Acal_{j,m}\cup\Acal_{m,m}^{1} & 1\leq m\leq l\\
\Acal_{j,m}\cup\dcal_{m-l-1,m-q}^{m} & m>l\end{cases}\]
$M(\mathbb{P}^{q,l,j})=\{\{F\cup\{x_{1},x_{2},\ldots,x_{q}\},\{r(\xi),y_{2},\ldots,y_l\},$

$\{z_{1},z_{2},\ldots,z_{j}\}\}_{(x_{1},x_{2},\ldots,x_{q})\in X^{q},(y_{2},
\ldots,y_{l})\in X^{l-1},(z_{1},z_{2},\ldots,z_{j})\in X^{j},F\in\xi}\}{}_{\xi\in\Phi}$
\item
$M(\{\Acal_{j,m}^{1}\cup\{\vec{m}\}\}_{m\in\mathbb{N}})=\{\{X,\{x_{1},x_{2},\ldots,x_{j}\}_{(x_{2},\ldots,x_{j})\in X^{j-1}}\}\}_{x_{1}\in X},\,(j\in\mathbb{N}).$
\item
$M(\{\Acal_{j,m}^{1}\cup\phi_{m}\}_{m\in\mathbb{N}})=\{\{\{r(\xi),x_{2},\ldots,x_{j}\},F\}_{(x_{2},\ldots,x_{j})\in X^{j-1},F\in\xi}\}_{\xi\in\Phi}\,(j\in\mathbb{N}).$
\item
$M(\{\Acal_{j,m}^{1}\cup\Acal_{j',m}\}_{m\in\mathbb{N}})=\{\{\{y_{1},y_{2},\ldots,y_{j'}\},$

$\{x_{1},x_{2},\ldots,x_{j}\}_{(x_{2},\ldots,x_{j})\in X^{j-1},(y_{1},y_{2},\ldots,y_{j'})\in X^{j'}}\}\}_{x_{1}\in X}\,(j,j'\in\mathbb{N})$.
\item
$M(\{\Acal_{j,m}^{1}\cup\Acal_{j',m}\cup\{\vec{m}\}\}_{m\in\mathbb{N}})=\{\{X,\{y_{1},y_{2},\ldots,y_{j'}\},$

$\{x_{1},x_{2},\ldots,x_{j}\}_{(x_{2},\ldots,x_{j})\in X^{j-1},(y_{1},y_{2},\ldots,y_{j'})\in X^{j'}}\}\}_{x_{1}\in X}$.
\item \label{enu:last usl}
$M(\{\Acal_{j,m}^{1}\cup\Acal_{j',m}\cup\phi_{m}\}_{m\in\mathbb{N}})=\{\{\{y_{1},y_{2},\ldots,y_{j'}\},$

$\{r(\xi),x_{2},\ldots,x_{j}\},F\}_{(x_{2},\ldots,x_{j})\in X^{j-1},F\in\xi,(y_{1},y_{2},\ldots,y_{j'})\in X^{j'}}\}_{\xi\in\Phi}$.
\item \label{enu:case W}
$M(\{\Acal_{m,m}\}_{m\in\mathbb{N}})=\{\{F|\, F\in Exp(X)\}\}$.
\item \label{enu:case W_X}
$M(\{\Acal_{m,m}^{1}\}_{m\in\mathbb{N}})=W_{X}=\{\{F|\, x\in F\in Exp(X)\}\}_{x\in X}$.
\item \label{enu:case W_X_j}
$M(\{\Acal_{m,m}^{1}\cup\Acal_{j,m}\}_{m\in\mathbb{N}})=\{\{\{x_{1},x_{2},\ldots,x_{j}\},F\}_{(x_{1},x_{2},\ldots,x_{j})\in X^{j},x\in F\in Exp(X)}\}_{x\in X}$
$(j\in\mathbb{N}).$
\end{enumerate}
\end{defn}
\begin{thm}\label{main2}
\label{thm:all minimal subspaces E_1}The standard minimal spaces
are the only minimal subspaces of $E_{1}$. \end{thm}
\begin{proof}
Let $M\subset E_{1}$ be a minimal subspace and $\xi\in M$. By Theorem
\ref{thm:Minimal Space contains stable signature} $sig(\xi)$ is
hereditary stable. For a subspace $N\subset E_{1}$ and $\alpha\in\D$
denote $U(N,\alpha)=\bigcup_{\xi\in N}U(\xi,\alpha)$.
Fix $\alpha\in\dcal_{m}$ with $m\geq3$. By Lemma \ref{lem:sig(xi)=00003D00003Dxi},
$N=\bigcap_{\beta\in S}U(N,\beta)$
for any cofinal set
$S\subset\D$, therefore it is enough to to show for any $\beta\succeq{\alpha}$
that $\xi_{\beta}\in U(N,\beta)$ for $N$ a standard
space. %
By the definition of hereditary stability
there exists $\sigma_{1}\in S_{m}$ so that $\pcal_{(\xi,\sigma_{1}({\alpha}))}\triangleq\pcal$
is hereditary stable. Assume first $\pcal$ has usl
(see Definition \ref{usl}). Let $\mathbb{P}=\{\pcal_{k}\}_{k\in\mathbb{N}}$
be the unique pattern-family so that
$\pcal_{m}=\pcal$
(it corresponds
to one of the first \ref{enu:last usl} cases in Definition \ref{def:Standard spaces}).
Let $\beta\succeq{\alpha}$ with $\beta\in\D_{m'}$.
Again there exists $\sigma_{2}\in S_{m'}$ so that $\pcal_{(\xi,\sigma_{2}({\beta}))}$
is hereditary stable. Let $\gamma\in\Pi{m' \choose m}$ and $\sigma_{3}\in S_{m}$
so that $(\sigma_{2}(\beta))_{\gamma}=\sigma_{3}\sigma_{1}(\alpha)$.
Conclude by Lemma \ref{lem:Pattern_Transformation}
that  $(\pcal_{(\xi,\sigma_{2}({\beta}))})_{\gamma}=\pcal_{(\xi,\sigma_{3}\sigma_{1}(\alpha))}=\sigma_{3}^{-1}(\pcal_{m})$.
As $\pcal_{m}$ is permutation stable by Theorem \ref{thm:hereditary is permutation stable},
conclude $(\pcal_{(\xi,\sigma_{2}(\beta))})_{\gamma}=\pcal_{m}$.
As $\pcal_{m}$ has usl one has that $\pcal_{(\xi,\sigma_{2}(\beta))}=\pcal_{m'}$,
i.e. $\pcal_{(\xi,\beta)}=\sigma_{2}^{-1}(\pcal_{m'})$, which
implies $\xi_{\beta}\in U(M(\mathbb{P}),\beta)$.

We now
treat the case when $\pcal$ does not have usl. Considering Theorem
\ref{thm:the only hereditary} and the proof of Theorem \ref{thm:Standard_Examples_are_Stable}
it is clear that $\pcal=\Acal_{m,m}^{1}$ or $\pcal=\acal_{m,m}$
or $\pcal=\Acal_{m,m}^{1}\cup\acal_{j,m}$ for $1\leq j\leq m-2$.
As in the case that $\pcal$ has usl, we have $(\pcal_{(\xi,\sigma_{2}(\beta))})_{\gamma}=\pcal_{m}$
for all $\gamma\in\Pi{m' \choose m}$\@. Let us consider the case
$m'=m+1$. Assume w.lo.g. $\pcal=\acal_{m,m}$ and denote $\mathbb{P}=\{\Acal_{m,m}\}_{m\in\mathbb{N}}$.
By Lemma \ref{lem:non usl cases}, $p_{\pi}^{-1}(\acal_{m,m})\cap HSP_{1}(m+1)=\{\Acal_{m,m+1},\Acal_{m+1,m+1},\Acal_{m+1,m+1}\cup\{\vec{m+1}\},\Acal_{m,m+1}^{1},\Acal_{m-1,m+1}\}\cup\{\dcal_{1,l}^{m+1}\cup\acal_{m-1,m+1}\}_{l\in\{2,\ldots,m+1\}}.$
Note that except for $\acal_{m+1,m+1}$ all members in the list have usl.
This implies $\pcal_{(\xi,\sigma_{2}(\beta))}$ has usl or $\pcal_{(\xi,\sigma_{2}(\beta))}=\acal_{m+1,m+1}$.
This means that either we have reduced to the usl case or $\xi_{\beta}\in U(M(\mathbb{P}),\beta)$.
Using induction we see that only three more types of minimal subspaces
are possible, namely the three last cases of the list in the statement
of the theorem.\end{proof}
\begin{thm}
The only minimal spaces of $E_{1}$ up to isomorphism are $\{\ast\}$,
$X$ and $\Phi$.\end{thm}
\begin{proof}
The result follows easily from Theorem \ref{thm:all minimal subspaces E_1}.
\end{proof}

\section{An Open Question}

In view of Theorem \ref{thm:Minimal Space contains stable signature}
the following problem is interesting:
\begin{problem}
Classify all hereditary $m$-stable $m_{n}$ patterns for $n\geq2$
and $m\in\mathbb{N}$.
\end{problem}
\bibliographystyle{alpha}
\bibliography{C:/Users/yonatan/Documents/Math/Bib/universal_bib}

\def\cprime{$'$}
\begin{thebibliography}{HNV04}

\bibitem[BS81]{BS81}
Stanley Burris and H.~P. Sankappanavar.
\newblock {\em A course in universal algebra}, volume~78 of {\em Graduate Texts
  in Mathematics}.
\newblock Springer-Verlag, New York, 1981.

\bibitem[BV80]{BV80}
Bohuslav Balcar and Peter Vojt{\'a}{\v{s}}.
\newblock Almost disjoint refinement of families of subsets of {${\bf N}$}.
\newblock {\em Proc. Amer. Math. Soc.}, 79(3):465--470, 1980.

\bibitem[dV93]{dV93}
J.~de~Vries.
\newblock {\em Elements of topological dynamics}, volume 257 of {\em
  Mathematics and its Applications}.
\newblock Kluwer Academic Publishers Group, Dordrecht, 1993.

\bibitem[Eng78]{E78}
Ryszard Engelking.
\newblock {\em Dimension theory}.
\newblock North-Holland Publishing Co., Amsterdam, 1978.
\newblock Translated from the Polish and revised by the author, North-Holland
  Mathematical Library, 19.

\bibitem[GG12]{GG12}
Eli Glasner and Yonatan Gutman.
\newblock The universal minimal space for groups of homeomorphisms of
  h-homogeneous spaces.
\newblock In {\em Dynamical Systems and Group Actions}, volume 567 of {\em
  Contemp. Math.}, pages 105--118. Amer. Math. Soc., Providence, RI, 2012.

\bibitem[GJ60]{GJ60}
Leonard Gillman and Meyer Jerison.
\newblock {\em Rings of continuous functions}.
\newblock The University Series in Higher Mathematics. D. Van Nostrand Co.,
  Inc., Princeton, N.J.-Toronto-London-New York, 1960.

\bibitem[GP84]{GP84}
Karsten Grove and Gert~Kj{\ae}rg{\aa}rd Pedersen.
\newblock Sub-{S}tonean spaces and corona sets.
\newblock {\em J. Funct. Anal.}, 56(1):124--143, 1984.

\bibitem[GR71]{GR71}
R.~L. Graham and B.~L. Rothschild.
\newblock Ramsey's theorem for {$n$}-parameter sets.
\newblock {\em Trans. Amer. Math. Soc.}, 159:257--292, 1971.

\bibitem[Gut08]{G08}
Yonatan Gutman.
\newblock Minimal actions of homeomorphism groups.
\newblock {\em Fund. Math.}, 198(3):191--215, 2008.

\bibitem[HNV04]{HNV04}
Klaas~Pieter Hart, Jun-iti Nagata, and Jerry~E. Vaughan, editors.
\newblock {\em Encyclopedia of general topology}.
\newblock Elsevier Science Publishers B.V., Amsterdam, 2004.

\bibitem[HS98]{HS98}
Neil Hindman and Dona Strauss.
\newblock {\em Algebra in the {S}tone-\v {C}ech compactification}, volume~27 of
  {\em de Gruyter Expositions in Mathematics}.
\newblock Walter de Gruyter \& Co., Berlin, 1998.
\newblock Theory and applications.

\bibitem[Kun80]{K80}
K.~Kunen.
\newblock Weak {$P$}-points in {${\bf N}^{\ast} $}.
\newblock In {\em Topology, {V}ol. {II} ({P}roc. {F}ourth {C}olloq.,
  {B}udapest, 1978)}, volume~23 of {\em Colloq. Math. Soc. J\'anos Bolyai},
  pages 741--749. North-Holland, Amsterdam, 1980.

\bibitem[OP89]{OP89}
Catherine~L. Olsen and Gert~K. Pedersen.
\newblock Corona {$C^*$}-algebras and their applications to lifting problems.
\newblock {\em Math. Scand.}, 64(1):63--86, 1989.

\bibitem[{\v{S}}R89]{SR89}
Petr {\v{S}}t{\v{e}}p{\'a}nek and Matatyahu Rubin.
\newblock Homogeneous {B}oolean algebras.
\newblock In {\em Handbook of {B}oolean algebras, {V}ol.\ 2}, pages 679--715.
  North-Holland, Amsterdam, 1989.

\bibitem[Usp09]{U09}
Vladimir Uspenskij.
\newblock On extremely amenable groups of homeomorphisms.
\newblock {\em Topology Proc.}, 33:1--12, 2009.

\bibitem[vM84]{vM84}
Jan van Mill.
\newblock An introduction to {$\beta\omega$}.
\newblock In {\em Handbook of set-theoretic topology}, pages 503--567.
  North-Holland, Amsterdam, 1984.

\end{thebibliography}

\address{Eli Glasner, School of Mathematical Sciences, Tel Aviv University,
Ramat Aviv, Tel Aviv 69978, Israel.}

\email{glasner@math.tau.ac.il}

\address{Yonatan Gutman, Laboratoire d'Analyse et de Mathématiques Appliquées,
Université de Marne-la-Vallée, 5 Boulevard Descartes, Cité Descartes
- Champs-sur-Marne, 77454 Marne-la-Vallée cedex 2, France.}

\email{yonatan.gutman@univ-mlv.fr}
\end{document}